\documentclass[12pt,letterpaper]{amsart}

\usepackage{graphicx}                % For embedding figures and subfigures
\usepackage[tight,center]{subfigure} % For embedding figures and subfigures
\usepackage{enumerate}               % To support \begin{enumerate}[(i)]
\usepackage{color}                %For colors dummy.
\usepackage{psfrag}
\usepackage{url}
\usepackage{multicol}

\usepackage[margin=1.2in]{geometry}
\makeatletter
\newtheorem*{rep@theorem}{\rep@title}
\newcommand{\newreptheorem}[2]{%
\newenvironment{rep#1}[1]{%
 \def\rep@title{#2 \ref{##1}}%
 \begin{rep@theorem}}%
 {\end{rep@theorem}}}
\makeatother

\newtheorem{theorem}{Theorem}
\newreptheorem{theorem}{Theorem}

\newtheorem{algorithm}[theorem]{Algorithm}

\theoremstyle{definition}
\newtheorem{defn}[theorem]{Definition}
\newtheorem{remark}[theorem]{Remark}

\newcommand{\knotinfo}{\emph{KnotInfo}}

\newcommand{\R}{\mathbb{R}}
\newcommand{\regina}{\emph{Regina}}
\newcommand{\snappy}{\emph{SnapPy}}
\newcommand{\tri}{\mathcal{T}}

\newcommand{\co}{\colon\thinspace}

\DeclareMathOperator{\wt}{wt}

\begin{document}

\title[Computing closed essential surfaces in knot complements]%
    {Computing closed essential surfaces in \\ knot complements}
\author{Benjamin A.~Burton}
\address{School of Mathematics and Physics \\
         The University of Queensland \\
         Brisbane QLD 4072 \\
         Australia}
\email{bab@maths.uq.edu.au}
\author{Alexander Coward}
\address{School of Mathematics and Statistics \\
         The University of Sydney \\
         Sydney NSW 2006 \\
         Australia}
\email{coward@maths.usyd.edu.au}
\author{Stephan Tillmann}
\address{School of Mathematics and Statistics \\
         The University of Sydney \\
         Sydney NSW 2006 \\
         Australia}
\email{tillmann@maths.usyd.edu.au}
\thanks{The authors are supported by the Australian Research Council
        under the Discovery Projects funding scheme (projects DP1094516
        and DP110101104).}

\date{}

\begin{abstract}
    We present a new, practical algorithm to test whether a
    knot complement contains a closed essential surface.
    This property has important theoretical and algorithmic consequences;
    however, systematically testing it has until now been infeasibly slow,
    and current techniques only apply to specific families of knots.
    As a testament to its practicality, we run the algorithm over
    a comprehensive body of 2979 knots, including
    the two 20-crossing dodecahedral knots,
    yielding results that were not previously known.
    
    The algorithm derives from the original Jaco-Oertel framework, involves
    both enumeration and optimisation procedures, and
    combines several techniques from normal surface theory.
    This represents substantial progress in the practical implementation of
    normal surface theory, in that we can systematically solve
    a theoretically double exponential-time problem for significant inputs.
    Our methods are relevant for other difficult computational
    problems in 3-manifold theory, ranging from testing for
    Haken-ness to the recognition problem for knots, links and 3-manifolds.
\end{abstract}

%%%%%%%%%%%%%%%%%%%%%%%%%%%%%%%%%%%%%%%%%%%%%%%%%%%%%%%%%%%%%%%%%%%%%%%%
%
%   Main body of the paper
%
%%%%%%%%%%%%%%%%%%%%%%%%%%%%%%%%%%%%%%%%%%%%%%%%%%%%%%%%%%%%%%%%%%%%%%%%

\maketitle

\section{Introduction}

%In 3-manifold theory, closed essential surfaces are of great
%importance. They may be thought of as `topologically minimal', in the
%sense that they may not be simplified to a surface of larger
%Euler-characteristic by compressing. 
In the study of 3--manifolds, essential surfaces have been of central
importance since Haken's seminal work in the 1960s. 
An essential surface may be regarded as `topologically minimal', and there has since been extensive research  into 3-manifolds, called \emph{Haken 3-manifolds}, that contain an essential surface. The existence of such a surface has
profound consequences for both the topology and geometry of a 3--manifold
\cite{haken68-overview, jaco79-jsj-book,
johannson79-book, johannson79-mapping,
thurston86-deformation, waldhausen68-large}.

% Ben: I put back the definition of Haken since we use it later on.

Given any closed 3-manifold,
specified by a triangulation, it is a theorem of Jaco and Oertel
\cite{jaco84-haken} from 1984 that one may algorithmically test for the
existence of a closed essential surface. However, their algorithm
has significant intricacies and is of double-exponential complexity in terms of the input size,
putting it well beyond the scope of a practical implementation. 

In this paper we present for the first time a practical algorithm that
can systematically test a significant class of 3-manifolds
for the existence of a closed essential surface, and which is both
efficient in practice and always conclusive.
To illustrate its power, we run this algorithm over a comprehensive
body of input data, yielding computer proofs of new mathematical
results.

The 3-manifolds we examine in this
paper are the motivating spaces for 3-manifold theory: knot complements.
These are the spaces that arise by removing a knotted curve from
3-dimensional space, although our methods can be extended to apply
to a far wider class of 3-manifolds. See the full version of this paper
for  generalisations.  In this paper we work with two collections
of input data.  First, for
each of the 2977 non-trivial prime knots that can be drawn with a diagram of at most 12
crossings, we determine whether its complement contains a closed essential
surface. If there is no such surface, the knot is called \emph{small},
otherwise we call it \emph{large}.  Second, we apply our algorithm to
resolve, in the affirmative, a question of Michel Boileau
\cite{boileau12-dodec} who enquired whether two special 20-crossing knots called
\emph{dodecahedral knots} contain a closed essential surface in their
complements. This question was recently also independently resolved by Jessica
Banks \cite{banks12-dodec} using non-computational techniques.

The algorithm presented here is theoretically significant because it is
the first algorithm in the literature for testing largeness of arbitrary
knots.  However, more important is its practical significance: this is the
first conclusive algorithm of this type that is implemented and fast enough
for real-world use.  The prior state-of-the-art algorithm for detecting
essential surfaces was used to prove that the Weber-Seifert
dodecahedral space is Haken \cite{burton12-ws}; however, although it resolved a
long-standing open problem, this prior algorithm relies on heuristic
methods that only work for certain triangulations, and are only
conclusive if no essential surface exists.  In contrast, the algorithm
described here can work with arbitrary triangulations of knot
complements, and is found to be effective regardless of the final result.

Our methods can be applied to related invariants of knots and 3-manifolds.
For instance,
the smallest genus $g$ of a closed essential surface is an important
knot invariant about which little is known for the case $g \geq 2$,
%.  If there is no such surface, then the knot $K$ is a torus
%knot or a small hyperbolic knot.  If $g=1,$ then its complement has a
%non-trivial JSJ decomposition. If $2\le  g < \infty,$ then $K$ is a
%large hyperbolic knot.  Beyond this, little is known about the
%topological and geometrical implications of the value of $g\ge 2,$
and our algorithm opens the door to formulating and testing new hypotheses.
These methods may also be extended to test a wide variety of
3-manifolds for Haken-ness and related properties.
More broadly, iterated
exponential complexity algorithms arise frequently in 3-manifold theory,
and our methods give an outline for how such problems, like the
recognition problem for knots and 3-manifolds, may one day be within the
realm of a practical implementation. 

We base our work on the framework of the Jaco-Oertel algorithm for
testing for closed incompressible surfaces.  This uses
\emph{normal surfaces}, which allow us to translate topological
questions about surfaces into the setting of integer and linear programming.
The framework consists of two stages: the first constructs a finite list of
candidate essential surfaces, and the second tests each surface in the
list to see if it is essential. A key difficulty with this framework,
which our algorithm also inherits,
is that both stages have running times that are
worst-case exponential in their respective input sizes, and
combining them in any obvious way leads to a double-exponential
complexity solution.

Despite this significant hurdle, we introduce several innovations that cut
down the running time enormously for both stages. Our optimisation for the
first stage involves a combination of established techniques that, though well
understood individually, require new ideas and theory in order to
work harmoniously together.
For the second stage we combine branch-and-bound techniques from
integer programming with the Jaco-Rubinstein procedure for crushing
surfaces within triangulations.
% The innovations for stage two are of particular significance,
% since there has never before been a systematic algorithm for
% testing whether a candidate surface is essential that is both practical
% and always conclusive.
In more detail:

\begin{itemize}
\item
For the {first stage} (enumerating candidate essential surfaces),
we combine several techniques.
First, we wish to create a triangulation for
each knot complement that contains as few tetrahedra as possible. If one
uses classical triangulations one needs as many as $50$ tetrahedra for
some knots in the 12-crossing tables, a size for which enumerating candidate
surfaces is thoroughly infeasible even for modern high-performance machines.
We therefore use \emph{ideal triangulations} for knot
complements, which are decompositions of these spaces into tetrahedra
with their vertices removed. These introduce some significant
theoretical difficulties, but they are much smaller with roughly half as
many tetrahedra.

Second, we use a variant of normal surface theory based on
\emph{quadrilateral coordinates}.
The appeal is that this brings the dimension of the underlying integer
and linear programming problems down from $7t$ in the classical setting
to $3t$, where $t$ is the number tetrahedra in the input 3-manifold.
These coordinates were known to Thurston and Jaco
in the 1980s, and first appeared in print in work of
Tollefson~\cite{tollefson98-quadspace}.

A theoretical difficulty arises when combining ideal triangulations with
quadrilateral coordinates: this introduces objects called
\emph{spun-normal surfaces}, which are properly embedded non-compact surfaces
(essentially built from infinitely many pieces).  We counter this
by introducing extra linear constraints called \emph{boundary equations}
which, with the development of appropriate theory, restrict the solution space
in question to closed surfaces only.  In particular,
using an extension of the work of Jaco and Oertel~\cite{jaco84-haken} from
compact manifolds to non-compact manifolds by Kang~\cite{kang05-spun},
we show in Theorem~\ref{thm:some extremal is closed essential}
that for each manifold under consideration, there is
a finite, constructible set of normal surfaces with the property that if
the manifold in question contains a closed essential surface, then one
must exist in this set.

\item
For the {second stage} (testing whether a candidate surface is essential),
the Jaco-Oertel approach cuts along each candidate surface and
inspects the boundary of the resulting 3-manifold to see if it admits a
\emph{compression disc} (such a disc certifies that a surface is non-essential).
The key difficulty is that one requires a new triangulation for the cut-open
3-manifold: since the candidate surface may be very complicated,
any natural scheme for cutting and re-triangulating yields a new triangulation
with exponentially many tetrahedra in the worst case, taking us far
beyond the realm in which normal surface theory has traditionally been
feasible in practice.
Since these new triangulations are the input for stage two,
which is itself exponential time, we now see where the double
exponential arises, and why the Jaco-Oertel framework has long been
considered far from practical. 

We resolve this significant problem using a blend of techniques.
First, we use strong simplification heuristics to reduce the number of
tetrahedra.  Next, we replace the traditional (and very expensive)
enumeration-based search for compression discs with an
\emph{optimisation} process that maximises Euler characteristic.
This uses the branch-and-bound techniques of \cite{burton12-unknot},
and allows us to quickly focus on a single
candidate compression disc.  We employ the crushing techniques
of Jaco and Rubinstein \cite{jaco03-0-efficiency} to quickly test
whether this is indeed a compression disc, and (crucially) to reduce the
size of the triangulation if it is not.

More generally, this issue of iterated-exponential complexity, coming from
cutting and re-triangulating, arises with ubiquity when considering objects
called \emph{normal hierarchies}.  These hierarchies are key when solving
more difficult problems such as the recognition problem for knots and
3-manifolds. Our approach to stage two is both fast in practice and
theoretically correct, making it a substantial breakthrough that
indicates that a practical implementation of these more difficult
algorithms might indeed be possible.
\end{itemize}

\section{Preliminaries}

\subsection{Knots, surfaces and triangulations}

A 3-manifold is a mathematical object that locally looks like
3-dimensional Euclidean space. Because every topological 3-manifold admits
precisely one piecewise-linear structure (up to PL-homeo\-morph\-ism)
\cite{moise52-triang},
in practice this means that 3-manifolds may be studied via
triangulations. A \emph{triangulation} of a compact 3-manifold $M$ is a
description of $M$ as the disjoint union of a finite collection of
3-simplices with their faces identified in pairs, as shown in
Figure~\ref{glueth}.

\begin{figure}[h!]
\centering
\includegraphics[scale=1]{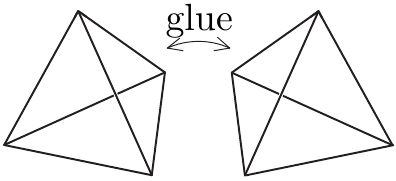}
\caption{A 3-manifold may be specified by a triangulation.} \label{glueth}
\end{figure}

A triangulation $\tri$ for a 3-manifold $M$ gives rise to
\emph{vertices}, \emph{edges}, \emph{faces} and \emph{tetrahedra} in
$M$. Edges whose
interior lies in the interior of $M$ are called \emph{interior
edges}, and edges that lie entirely on the boundary of $M$ are called
\emph{boundary edges}. In practice, a tetrahedron in $M$ might not be
embedded; for instance, we even allow two faces of a tetrahedron to be
identified in $M.$ For a precise description of our set-up, please see
\S\ref{subsec:Triangulations}; and for an example, see Appendix~\ref{app:fig8}.

Such a triangulation can only specify a compact
3-manifold. If instead of identifying the faces of tetrahedra, we
identify the faces of a finite collection of tetrahedra minus their
vertices, this constitutes an \emph{ideal triangulation} for the
resulting non-compact quotient space.

The 3--manifolds we study in this paper are \emph{knot complements}.
These are 3-manifolds obtained by removing a \emph{knot}, which is
knotted closed curve in $\R^3$, from 3--dimensional Euclidean space. In
practice it is convenient to compactify $\R^3$ with a point at infinity,
yielding a compact 3-manifold called the \emph{3-sphere}, denoted $S^3.$
One then removes the knot from $S^3$ instead. For a knot $K$
we call the resulting non-compact 3-manifold $S^3 \backslash K$ the
% \emph{complement of $K$ in $S^3$}, or simply the
\emph{complement of $K$}.
Knot complements always have ideal triangulations
\cite[Proposition~1.2]{tillmann08-finite}.

If instead we remove from $S^3$ a small
open neighborhood of a knot $K$
we obtain a compact 3-manifold called the
% \emph{exterior of $K$ in $S^3$}, or simply the
\emph{exterior of $K$}. Since they are compact, knot exteriors may be
specified by triangulations. There are well established techniques for
translating between an ideal triangulation for a knot complement and a
triangulation for the corresponding knot exterior. 

A knot $K$ is called \emph{non-trivial} if it is not the boundary of an
embedded disc in $S^3.$

In this paper we are interested in \emph{closed essential surfaces} in knot
complements. We define these now. Let $K$ be a knot, and let $M$ be the
complement of $K$. A connected two-sided closed surface with
positive genus $S$, embedded in $M$, is a \emph{closed essential surface
in $M$} if  the following properties hold: (i)~the surface $S$ is
\emph{incompressible} (as defined below); and (ii)~the surface $S$ is not
\emph{boundary parallel}, that is, not ambient isotopic to a small tube
running around $K.$

The definition of incompressible is as follows. A \emph{compression
disc} for an embedded surface $S$ in a 3-manifold $M$ is an embedded
disc $D\subset M$ for which (i)~$D \cap S$ equals the boundary of $D$
(denoted $\partial D$); and (ii)~$\partial D$ is a non-trivial curve in $S$ (meaning $\partial D$ does not bound a disc in $S$). If the surface $S$ admits
a compression disc, then we say $S$ is \emph{compressible},
otherwise $S$ is \emph{incompressible}.  An equivalent, algebraic
criterion can be found in \S\ref{subsec:essential}.

% The reason that incompressible surfaces are so important in 3-manifold
% theory is that they may not be simplified via a \emph{compression}. This
% is a special type of move best imaged by considering what happens when
% two circular hoops are dipped close together into soap solution,
% yielding an annulus  shaped soap film, and then the two hoops are moved
% apart until the soap film resolves into two discs. The importance of
% this class of surfaces in 3-manifold theory cannot be overstated.  More
% generally, we say that a disconnected, properly embedded surface in $M$
% is essential if each of its connected components is essential.

\subsection{Quadrilateral coordinates and $Q$--matching equations}

We use normal surface theory to search for essential
surfaces. A \emph{normal surface} in a (possibly ideal) triangulation
$\tri$ is a properly embedded surface which intersects each
tetrahedron of $\tri$ in a disjoint collection of
\emph{triangles} and \emph{quadrilaterals}, as shown in Figure
\ref{normaldiscs}. These triangles and quadrilaterals are called
\emph{normal discs}. In an ideal triangulation of a non-compact
3--manifold, a normal surface may contain infinitely many triangles;
such a surface is called \emph{spun-normal}.

\begin{figure}[h!]  
\centering
\includegraphics[scale=0.7]{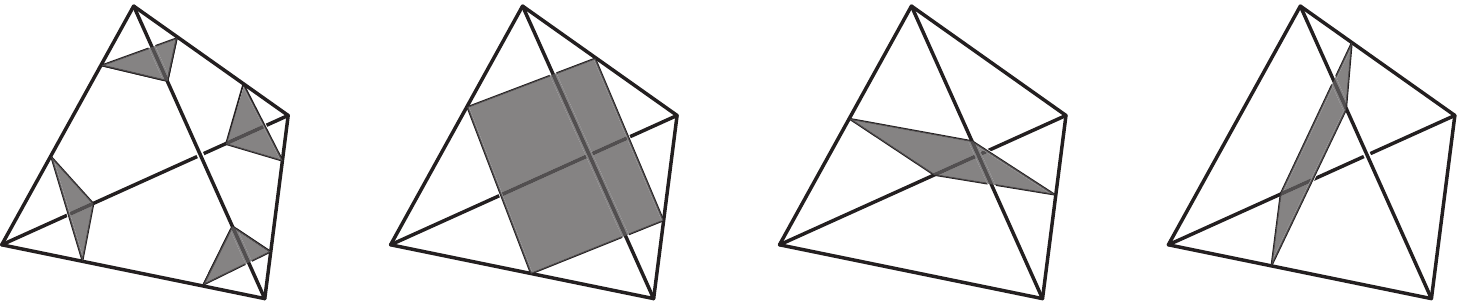}
\caption{The seven types of normal disc in a tetrahedron.} 
\label{normaldiscs}
\end{figure}

We now describe an algebraic approach to normal surfaces.
The key observation is that each normal surface contains
finitely many quadrilateral discs, and is uniquely determined
(up to normal isotopy) by these quadrilateral discs. Here a \emph{normal
isotopy} of $M$ is an isotopy that keeps all simplices of all dimensions
fixed. Let $\square$ denote the set of all normal isotopy classes,
or \emph{types}, of normal quadrilateral discs,
so that $|\square| = 3t$ where $t$ is the number of tetrahedra in $\tri$.
We identify $\R^\square$ with $\R^{3t}.$
Given a normal surface $S,$ let $x(S) \in \R^\square = \R^{3t}$
denote the integer vector for which each coordinate $x(S)(q)$
counts the number of quadrilateral discs in $S$ of type $q \in \square$.
This \emph{normal $Q$--coordinate} $x(S)$ satisfies the
following two algebraic conditions.

First, $x(S)$ is admissible.
A vector $x \in \R^\square$ is \emph{admissible} if
$x \ge 0$, and for each tetrahedron $x$ is non-zero
on at most one of its three quadrilateral types. 
This reflects the fact that an embedded surface cannot contain
two different types of quadrilateral in the same tetrahedron.

Second, $x(S)$ satisfies a linear equation for each interior
edge in $M,$ termed a \emph{$Q$--matching equation}.
% These equations
% were first described in print by Tollefson~\cite{tollefson98-quadspace}.
Intuitively, these equations arise from the fact that as one
circumnavigates the earth, one crosses the equator from north to south
as often as one crosses it from south to north.
We now give the precise form of these equations.
To simplify the discussion,
we assume that $M$ is oriented and all tetrahedra are given
the induced orientation; see \cite[Section~2.9]{tillmann08-finite} for
details.

\begin{figure}[h]
    \centering
    \subfigure[The abstract neighbourhood $B(e)$]{%
        \label{fig:matchingquadbdry}%
        \includegraphics[scale=1]{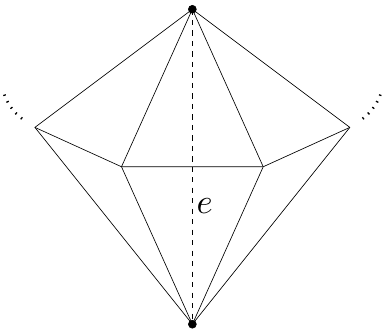}}
    \qquad
    \subfigure[Positive slope]{%
        \label{fig:matchingquadpos}%
        \includegraphics[scale=1]{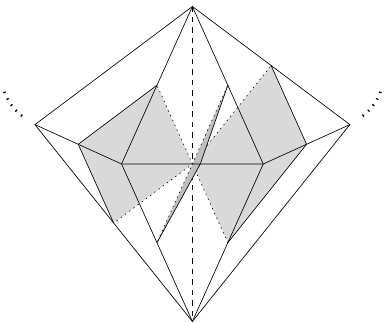}}
    \qquad
    \subfigure[Negative slope]{%
        \label{fig:matchingquadneg}%
        \includegraphics[scale=1]{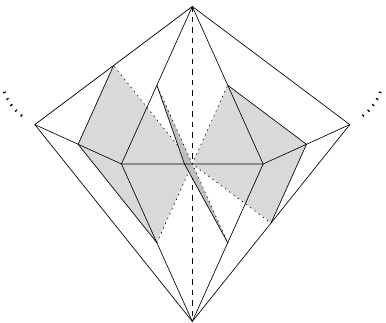}}
    \caption{Slopes of quadrilaterals}
    \label{fig:slopes}
\end{figure}

Consider the collection $\mathcal{C}$ of all (ideal) tetrahedra meeting
at an edge $e$ in $M$ (including $k$ copies of tetrahedron $\sigma$ if $e$ occurs
$k$ times as an edge in $\sigma$).  We form the \emph{abstract
neighbourhood $B(e)$} of $e$ by pairwise identifying faces of tetrahedra
in $\mathcal{C}$ such that there is a well defined quotient map from
$B(e)$ to the neighbourhood of $e$ in $M$; see
Figure~\ref{fig:matchingquadbdry} for an illustration.  Then  $B(e)$ is a ball
(possibly with finitely many points missing on its boundary).  We think
of the (ideal) endpoints of $e$ as the poles of its boundary sphere, and
the remaining points as positioned on the equator.

Let $\sigma$ be a
tetrahedron in $\mathcal{C}$. The boundary square of a normal
quadrilateral of type $q$ in $\sigma$ meets the equator of $\partial
B(e)$ if and only it has a vertex on $e$. In this case, it has a slope
of a well--defined sign on $\partial B(e)$ which is independent of the
orientation of $e$. Refer to Figures~\ref{fig:matchingquadpos} and
\ref{fig:matchingquadneg}, which show quadrilaterals with
\emph{positive} and \emph{negative slopes} respectively.

% I took out this proof of the Q-matching equations.. ok?
% This was to save space, and we have the "intutive reason" above anyway
% (travelling north/south).  - Ben.
%
% We observe that if $S$ is properly embedded, then $S$ intersects $B(e)$
% in (possibly infinitely many) topological discs which are made up of
% normal discs, and hence $S$ intersects $\partial B(e)$ in a collection
% of circles. Triangle discs with corners on $e$ do not meet the equator
% of $\partial B(e)$. Thus, the number of positive slopes of
% quadrilaterals in $S$ on $\partial B(e)$ must be equal to the number of
% negative slopes, since otherwise the intersection of some quadrilateral
% disc in $S$ with $\partial B(e)$ is not contained in a circle, but in a
% helix on $\partial B(e)$.

% Since our triangulations allow edges of a tetrahedron to be identified
% in $M,$
Given a quadrilateral type $q$ and an edge $e,$ there is a
\emph{total weight} $\wt_e(q)$ of $q$ at $e,$ which records the sum of
all slopes of $q$ at $e$ (we sum because $q$ might meet $e$ more than
once, if $e$ appears as multiple edges of the same tetrahedron).
If $q$ has no corner on $e,$ then we set
$\wt_e(q)=0.$ Given edge $e$ in $M,$ the $Q$--matching equation of $e$
is then defined by $0 = \sum_{q\in \square}\; \wt_e(q)\;x(q)$.

\begin{theorem}\label{thm:admissible integer solution gives normal}
For each $x\in \R^\square$ with the properties that $x$ has integral coordinates, $x$ is admissible and
$x$ satisfies the $Q$--matching equations, there is a (possibly
non-compact) normal surface $S$ such that $x = x(S).$ Moreover, $S$ is
unique up to normal isotopy and adding or removing vertex linking surfaces, i.e.
normal surfaces consisting entirely of normal triangles.
\end{theorem}

For a proof of Theorem \ref{thm:admissible integer solution gives normal} see Theorem 2.1 of \cite{kang05-spun} or Theorem 2.4 of
\cite{tillmann08-finite}.

The set of all $x\in \R^{\square}$ with the property that
(i)~$x\ge 0$ and (ii)~$x$ satisfies the $Q$--matching equations is denoted
$Q(\tri).$ This
naturally is a polyhedral cone.
Note however that the set of all admissible $x\in \R^{\square}$
typically meets $Q(\tri)$ in a non-convex set.

% Not used. - Ben.
% The intersection of $Q(\tri)$ with the hyperplane
% $\sum_{q\in \square} x(q) = 1$ is the \emph{projective solution space}
% $PQ(\tri).$

\subsection{Crushing triangulations}
\label{subsec:Crushing}

The crushing process of Jaco and Rubinstein
\cite{jaco03-0-efficiency} plays an important role in our
algorithms, and we informally outline this process here.
For the formal details we refer the reader to
\cite{jaco03-0-efficiency}.

Let $S$ be a two-sided normal surface in a triangulation $\tri$ of a compact
orientable 3-manifold $M$.  To \emph{crush} $S$ in $\tri$, we
(i)~cut $\tri$ open along $S$, which splits each tetrahedron
into a number of (typically non-tetrahedral) pieces;
(ii)~crush each resulting copy of $S$ on the boundary to a point,
which converts these pieces into tetrahedra, footballs and/or pillows
as shown in Figure~\ref{fig:jrpieces};
(iii)~flatten each football or pillow to an edge or triangle
respectively, as shown in Figure~\ref{fig:jrflatten}.

\begin{figure}[ht]
    \centering
    \subfigure[Pieces after crushing $S$ to a point]{%
        \label{fig:jrpieces}%
        \includegraphics[scale=0.5]{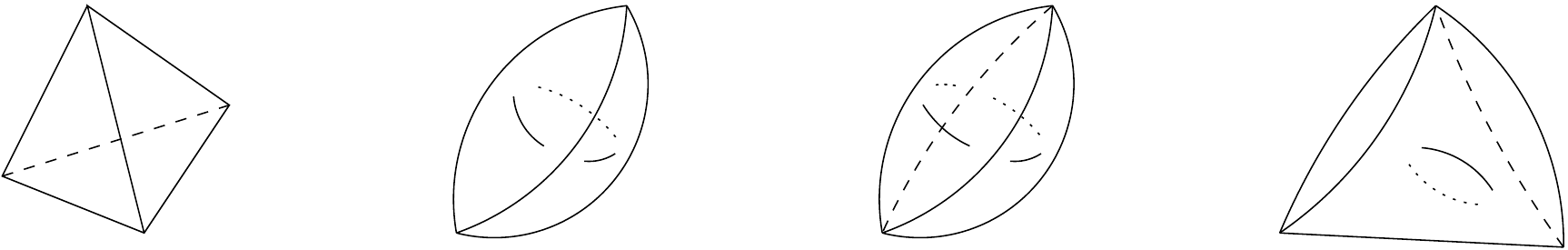}}
    \\
    \subfigure[Flattening footballs and pillows]{%
        \label{fig:jrflatten}%
        \includegraphics[scale=0.833]{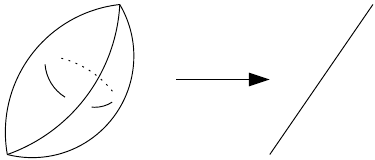}\quad
        \includegraphics[scale=0.833]{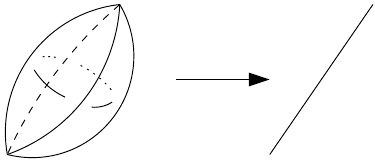}\quad
        \includegraphics[scale=0.833]{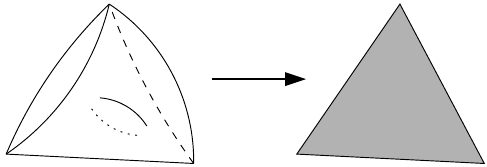}}
    \caption{Steps in the Jaco-Rubinstein crushing process}
    \label{fig-jrcrush}
\end{figure}

The result is a new collection of tetrahedra with a new set of face
identifications.  We emphasise that we \emph{only} keep track of face
identifications between tetrahedra:
any ``pinched'' edges or vertices fall apart, and any
lower-dimensional components with no tetrahedra at all simply disappear.
The resulting structure might not represent a 3-manifold
triangulation, and even if it does the flattening operations might have
changed the underlying 3-manifold in ways that we did not intend.

% Let $S$ be a normal surface in the triangulated manifold $M.$ For each vertex in the triangulation of $M,$ add a copy of the corresponding vertex linking surface to $S,$ giving a surface $F.$ Now $F$ cuts each tetrahedron into a number of different pieces:
% truncated tetrahedra;
% truncated prisms;
% slabs (bounded by two parallel triangles or quads);
% small tetrahedra with base a normal triangle and vertex a (possibly ideal) vertex in $M.$

% The small tetrahedra glue up in $M$ to give a small regular neighbourhood of the vertices; we therefore understand their topology very well, and we are only interested in the remaining components of $M\setminus F.$ For each such component $X$ of $M\setminus F,$ one \emph{would like to} obtain a canonical ideal triangulation by collapsing to a point each boundary component of $X$ corresponding to a component of $S.$ In doing so, we will crush each 
% truncated tetrahedron to a tetrahedron, each truncated prism to a triangle, and each slab (bounded by two parallel triangles or quads) to an edge.

% In particular, the tetrahedra in the \emph{desired} triangulation of $X$ arise from the truncated tetrahedra in $X,$ and the corresponding face pairings arise from the natural identification of one hexagon in the boundary of a truncated tetrahedron with another hexagon in the boundary of a truncated tetrahedron, which is obtained by possibly following through a sequence of truncated prisms.

Although crushing can cause a myriad of problems in general,
Jaco and Rubinstein show that in some cases the operation behaves
extremely well \cite{jaco03-0-efficiency}.  In particular, if
$S$ is a normal sphere or disc, then after crushing we always obtain a
triangulation of some 3-manifold $M'$ (possibly disconnected, and
possibly empty) that is obtained from the original
$M$ by zero or more of the following operations:
\begin{itemize}
\item cutting manifolds open along spheres and filling the resulting boundary spheres with
3-balls;
\item cutting manifolds open along properly embedded discs;
\item capping boundary spheres of manifolds with 3-balls;
\item deleting entire connected components that are any of the 3-ball, the
3-sphere,
projective space $\R P^3,$ the lens space $L(3,1)$ or the product space $S^2 \times S^1.$
\end{itemize}

An important observation is that the number of tetrahedra that remain
after crushing
is precisely the number of tetrahedra that do not contain quadrilaterals of $S$.

\section{Closed normal surfaces in $Q$--space}
\label{sec:closed}

In this section we introduce the linear \emph{boundary equations},
with which we restrict the normal surface solution space to
closed surfaces only.

Let our knot complement be $M = S^3\setminus K.$ The ideal triangulation $\tri$ of $M$ has one ideal vertex, and its link is a torus. We view this torus $T$ as made up of normal triangles, one near each corner of each ideal tetrahedron. Let $\gamma \in H_1(T; \R).$ We now describe an associated linear functional $\nu(\gamma) \co \R^{\square}\to \R,$ which measures the behaviour along $\gamma$ of a normal surface near the ideal vertex. The idea is similar to the intuitive description of the $Q$--matching equations. As one goes along $\gamma$ and looks down into the manifold, normal quadrilaterals will (as Jeff Weeks puts it) \emph{come up from below} or \emph{drop down out of sight}. If the total number coming up minus the total number dropping down is non-zero, then the surface spirals towards the knot in the cross section $\gamma \times [0, \infty) \subset T \times [0, \infty)$ and the sign indicates the direction, see Figure~\ref{fig:spinning}(b). If this number is zero, then after a suitable isotopy the surface meets the cross section in a (possibly empty or infinite) union of circles, see Figure~\ref{fig:spinning}(c).

The torus $T$ has an induced triangulation consisting of normal triangles. Represent $\gamma$ by an oriented path on $T,$ which is disjoint from the 0--skeleton and meets the 1--skeleton transversely. Each edge of a triangle in $T$ is a normal arc. Give the edges of each triangle in $T$ transverse orientations pointing into the triangle and labelled by the quadrilateral types sharing the normal arc with the triangle; see Figure~\ref{fig:quad coods}. We then define $\nu(\gamma)$ as follows. Choosing any starting point on $\gamma,$ we read off a formal linear combination of quadrilateral types $q$ by taking $+q$ each time the corresponding edge is crossed with the transverse orientation, and $-q$ each time it is crossed against the transverse orientation (where each edge in $T$ is counted twice---using the two adjacent triangles). 

\begin{figure}[h]
    \centering
    \includegraphics[scale=1.1]{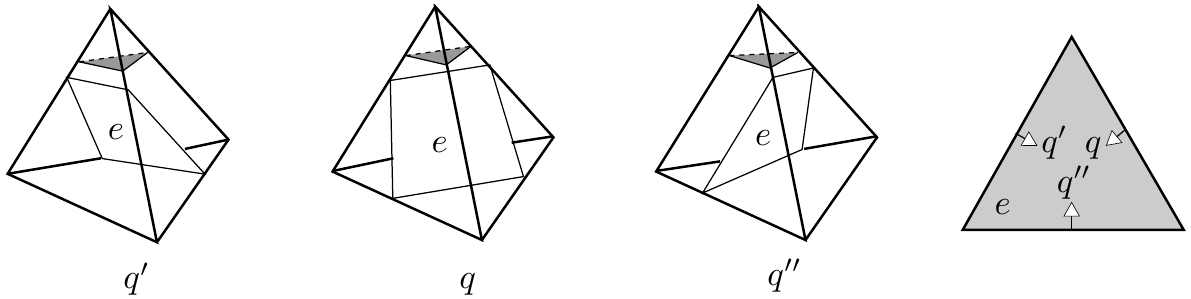}
    \caption{Coming up and dropping down}
    \label{fig:quad coods}
\end{figure}

Evaluating $\nu(\gamma)$ at some $x\in \R^{\square}$ gives a real number $\nu_{x}(\gamma).$ For example, taking a small loop around a vertex in $T$ and setting this equal to zero gives the $Q$--matching equation of the corresponding edge in $M;$ see Figure~\ref{fig:spinning}(a). For each $x\in Q(\tri),$ the resulting map $\nu_{x}\co H_1(T ; \R) \to \R$ is a well-defined homomorphism, which has the property that the surface in Theorem~\ref{thm:admissible integer solution gives normal} is closed if and only if $\nu_x = 0$ (see \cite{tillmann08-finite}, Proposition~3.3). Since $\nu_{x}\co H_1(T; \R) \to \R$ is a homomorphism, it is trivial if and only if we have $\nu_{x}(\alpha) = 0 = \nu_{x}(\beta)$ for any basis $\{\alpha, \beta\}$ of $H_1(T; \R).$ 

\begin{figure}[h]
  \begin{center}
    \subfigure[$0=\nu_x(\gamma) = \sum_{i=1}^{k} (-1)^i x(q_i)$ is the $Q$--matching equation]{
      \includegraphics[scale=1]{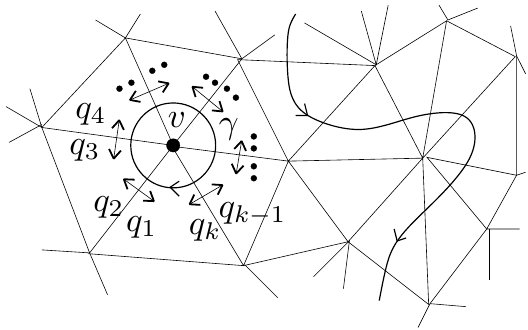}}
    \qquad
    \subfigure[spun $\Longleftrightarrow$ $\nu_x\neq 0$]{
      \includegraphics[scale=0.9]{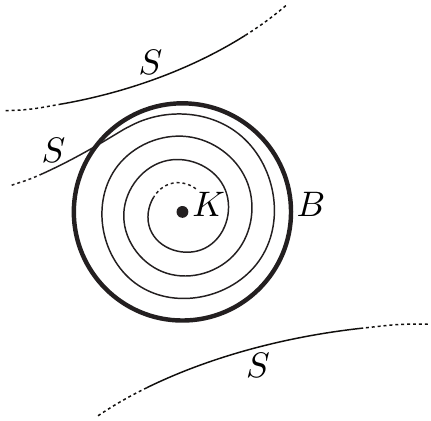}}
    \subfigure[not spun $\Longleftrightarrow$ $\nu_x= 0$]{
      \includegraphics[scale=0.9]{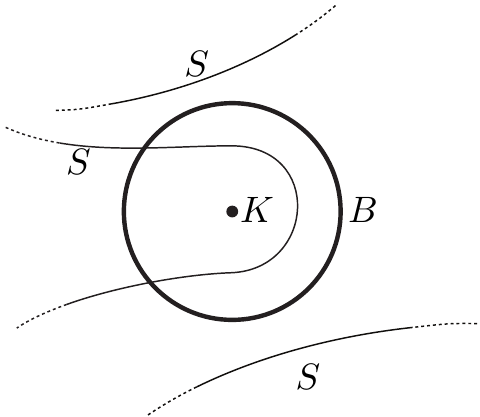}}
  \end{center}
  \caption{Boundary map determines $Q$--matching equations and spinning}
  \label{fig:spinning}
\end{figure}

We define $Q_0(\tri) = Q(\tri) \cap \{ x \mid \nu_x=0\},$ and call a two-sided, connected normal surface $F$ with $x(F)$ on an extremal ray of $Q_0(\tri)$ a \emph{$Q_0$--vertex surface.} The following result is based on the seminal work of  Jaco and Oertel~\cite{jaco84-haken}:

\begin{theorem}\label{thm:some extremal is closed essential}
Suppose $M$ is the complement of a non-trivial knot in $S^3.$ If $M$ contains a closed essential surface $S,$ then there is a normal, closed essential surface $F$ with the property that $x(F)$ lies on an extremal ray of $Q_0(\tri).$ Moreover, if $\chi(S)<0,$ then there is such $F$ with $\chi(F)<0.$ 
\end{theorem}

\begin{proof}[Sketch of proof] A complete proof of a more
general statement is given in Appendix~\ref{app:Jaco-Oertel}. The key ideas are as follows. Given a closed essential surface in $M,$
a standard argument shows that there is a \emph{normal} closed
essential surface in $M.$ Amongst all normal surfaces isotopic (but
not necessarily normally isotopic) to this, choose one that has
minimal number of intersections with the 1--skeleton (this is the PL
analogue of a minimal surface).  Denote this surface $S$.

If $S$ is not a vertex surface,
one can write it using a so-called \emph{Haken sum} of vertex surfaces,
which is a geometric realisation of the sum of $Q$--coordinate
vectors. However, a complication arises, since only a multiple of $S$ is
a Haken sum of vertex surfaces, and only up to
vertex linking tori; that is,
$nS + \Sigma = \sum n_iV_i = V + W,$ where $V$ is a
vertex surface, $\Sigma$ is vertex linking,
and all other terms are combined into the surface $W.$
Building on Jaco and Oertel~\cite{jaco84-haken},
Kang~\cite[Theorem~5.4]{kang05-spun} shows that $V$ and $W$ are
incompressible. Since Euler characteristic is additive under Haken sum,
the result follows if $\chi(S)<0.$ If $\chi(S)=0,$ additional work
is required to show that an essential torus cannot be written as a Haken
sum of boundary parallel tori.
\end{proof}

\section{Algorithms}
\label{sec:Detecting compressing discs}

Here we describe the new algorithm to test whether a knot is large or small
(i.e., whether its complement contains a closed essential surface).
In this extended abstract we restrict our attention to the common setting
of knots in the 3-sphere $S^3$.
See the full version of this paper
for an extension to the more general setting of links in arbitrary
closed orientable 3-manifolds, as well as searching for essential
surfaces in arbitrary closed orientable 3-manifolds (without knots or links).

We present the algorithm in two stages below.
Algorithm~\ref{alg-incompressible} describes a subroutine to test
whether a given closed surface is incompressible.
Algorithm~\ref{alg-large} is the main algorithm: it uses the results
of Section~\ref{sec:closed} to identify candidate essential surfaces, and
runs Algorithm~\ref{alg-incompressible} over each.

These algorithms contain a number of high-level and often intricate
procedures, many of which are described in separate papers.
For each algorithm, we discuss these procedures in further detail after
presenting the overall algorithm structure.

\begin{algorithm}[Testing for incompressibility]
\label{alg-incompressible}
Let $\tri$ be an ideal triangulation of a non-compact 3-manifold $M$
that is the complement of a non-trivial knot in $S^3.$
Let $S$ be a closed two-sided normal surface of genus $g \geq 1$
within $\tri$.  To test whether $S$ is incompressible in $M$:
\begin{enumerate}
    \item
    \label{en-alg-truncate}
    Truncate the ideal vertex of $\tri$ to obtain a compact manifold
    with boundary, cut $\tri$ open along the surface $S$, and retriangulate.
    The result is a pair of triangulations $\tri_1,\tri_2$ representing
    two compact manifolds with boundary $M_1,M_2$ (one on each side of
    $S$ in $M$).

    Let $B_1,B_2$ be the genus $g$ boundary components of $\tri_1$ and
    $\tri_2$ respectively that correspond to the surface $S$.
    Without loss of generality, suppose that the truncated
    ideal vertex was on the side of $M_2$; therefore $\tri_2$ has an
    additional boundary torus, which we denote $B_v$.

    \item
    For each $i=1,2$:
    \begin{enumerate}
        \item \label{en-alg-simplify}
        Simplify $\tri_i$ into a triangulation with no internal
        vertices and only one vertex on each boundary component,
        without increasing the number of tetrahedra.
        Let the resulting number of tetrahedra in $\tri_i$ be $n$.

        \item \label{en-alg-search}
        Search for a connected normal surface $E$ in $\tri_i$
        that is not a vertex link, has positive Euler
        characteristic, and (for the case $i=2$) does not
        meet the torus boundary $B_v$.

        \item \label{en-alg-nodisc}
        If no such $E$ exists, then there is no compressing disc
        for $S$ in $M_i$.  If $i=1$ then try $i=2$ instead, and if $i=2$
        then terminate with the result that $S$ is incompressible.

        \item \label{en-alg-crush}
        Otherwise, crush the surface $E$
        as explained in Section~\ref{subsec:Crushing}
        to obtain a new triangulation $\tri'_i$ (possibly disconnected,
        or possibly empty) with strictly fewer than
        $n$ tetrahedra.  If some component of $\tri'_i$ has the same
        genus boundary (or boundaries) as $\tri_i$ then it
        represents the same manifold $M_i$, and we return to
        step~\ref{en-alg-simplify} using this component of $\tri'_i$ instead.
        Otherwise we terminate with the result that $S$ is not incompressible.
    \end{enumerate}
\end{enumerate}
\end{algorithm}

\noindent
Regarding the individual steps of this algorithm:
\begin{itemize}
    \item Step~\ref{en-alg-truncate} requires us to truncate an ideal
    vertex and cut a triangulation open along a normal surface.
    These are standard (though intricate) procedures.
    To truncate a vertex we subdivide tetrahedra and then remove the
    immediate neighbourhood of the vertex.
    To cut along a normal surface, see \cite{burton12-ws} for a
    description of a manageable implementation.

    \item Step~\ref{en-alg-simplify} requires us to simplify
    a triangulation to use the fewest possible vertices, without
    increasing the number of tetrahedra.
    For this we begin with the rich polynomial-time simplification
    heuristics described in \cite{burton12-regina}.  In practice, for all
    $2979$ knots that we consider in Section~\ref{s-results},
    this is sufficient to reduce the triangulation
    to the desired number of vertices.

    If there are still extraneous vertices, we can remove these using
    the crushing techniques of Jaco and Rubinstein
    \cite[Section~5.2]{jaco03-0-efficiency}.  This might
    fail, but only if  $\partial M_i$ has  a compressing disc,
    or two boundary components of $M_i$ are separated by a sphere;
    both cases immediately certify that the surface $S$
    is compressible, and we can terminate immediately.

    \item Step~\ref{en-alg-search} requires us to locate a connected
    normal surface $E$ in $\tri_i$ that is not a vertex link, has
    positive Euler characteristic, and does not meet the torus
    boundary $B_v$.
    For this we use the recent method of
    \cite[Algorithm~11]{burton12-unknot}, which draws on
    combinatorial optimisation techniques:
    in essence we run a sequence of linear programs over a combinatorial
    search tree, and prune this tree using tailored branch-and-bound
    methods.  See \cite{burton12-unknot} for details.

    We note that this search is the bottleneck of
    Algorithm~\ref{alg-incompressible}:
    the search is worst-case exponential time, though in practice it
    often runs much faster \cite{burton12-unknot}.
    The exposition in \cite{burton12-unknot} works in the setting where
    the underlying triangulation is a knot complement, but the
    methods work equally well in our setting here.
    To avoid the boundary component $B_v$, we simply remove all normal
    discs that touch $B_v$ from our coordinate system.

    %\item Step~\ref{en-alg-crush} requires us to crush the surface $E$
    %in the triangulation $\tri_i$.
    %This uses the Jaco-Rubinstein ``destructive crushing'' procedure
    %\cite{jaco03-0-efficiency}, which
    %which is simple and fast and never increases the number of tetrahedra,
    %but which might make unwanted changes to the topology of the underlying
    %manifold.
    %
    %In this extended abstract we give just the properties of the
    %crushing process that are required to prove the algorithm correct
    %(see the proof below).  For more details on the crushing process,
    %see \cite[Section~3.1]{burton12-unknot} for a simple overview,
    %or \cite{jaco03-0-efficiency} for the full details.
\end{itemize}

\begin{theorem} \label{thm-incompressible}
    Algorithm~\ref{alg-incompressible} terminates, and its output is correct.
\end{theorem}

\begin{proof}
    The algorithm terminates because each time we loop
    back to step~\ref{en-alg-simplify} we have fewer tetrahedra
    than before.  Correctness is more interesting: there are many claims
    in the algorithm statement that require proof.
    Full proofs are given in Appendix~\ref{app-incompressible};
    the key ideas are as follows.
    \begin{itemize}
        \item In step~\ref{en-alg-truncate} we claim that cutting along
        $S$ yields two (disconnected) compact manifolds.
        This follows from the fact that 
        every closed surface embedded in the 3-sphere is separating.

        \item In step~\ref{en-alg-nodisc} we claim that, if the surface
        $E$ cannot be found in $\tri_1$ \emph{and} it cannot be found in
        $\tri_2$, then the original surface $S$ must be incompressible.
        This is because otherwise, by results of Jaco and Oertel
        \cite{jaco84-haken},
        there must be a \emph{normal} compressing disc on one side of $S$.

        \item In step~\ref{en-alg-crush} we make several claims.
        First, the new triangulation $\tri'_i$ has strictly fewer tetrahedra
        because $E$ is not a vertex link.
        Second, we claim that if $\tri'_i$ has a component with the
        same genus boundary (or boundaries) as $\tri_i$ then it represents
        the same manifold $M_i$, and otherwise $S$ is compressible;
        this is because the ``destructive'' side-effects of
        the crushing process reduce the boundary genus by cutting
        along compressing discs for $S$.
    \end{itemize}
    There are additional complications involving irreducibility;
    again see Appendix~\ref{app-incompressible} for details.
\end{proof}

% We can now package together a full algorithm to test for closed
% essential surfaces:

\begin{algorithm}[Testing whether a knot is large or small] \label{alg-large}
    Let $K$ be a non-trivial knot in $S^3$.
    To test whether $K$ is large or small:
    \begin{enumerate}
        \item \label{en-test-ideal}
        Build an ideal triangulation $\tri$ of the complement of $K$ in $S^3$.

        \item \label{en-test-enumerate}
        Enumerate all extremal rays of $Q_0(\tri)$;
        denote these $\mathbf{e}_1,\ldots,\mathbf{e}_k$.
        For each extreme ray $\mathbf{e}_i$,
        let $S_i$ be the unique connected two-sided normal
        surface for which $x(S_i)$ lies on $\mathbf{e}_i$.
        Ignore all surfaces $S_i$ that are spheres.

        \item \label{en-test-essential}
        For each remaining surface $S_i$,
        use algorithm~\ref{alg-incompressible} to test whether $S_i$
        is incompressible in $\tri$.
        If any $S_i$ is incompressible and is not a torus, then terminate
        with the result that $K$ is large.
        If no $S_i$ is incompressible, then terminate with the result that
        $K$ is small.

        \item \label{en-test-special}
        %Otherwise the only incompressible surfaces in our list
        %are tori, and there is at least one such incompressible torus.
        %Test whether the complement of $K$ contains an incompressible
        %\emph{non-boundary-parallel} torus using the algorithm of Kang
        %\cite[Remark~5]{kang05-spun}.  If it does then $K$ is large,
        %and if not then $K$ is small.
        Otherwise the only incompressible surfaces in our list
        are tori.  For each incompressible torus $S_i$, test whether
        $S_i$ is boundary parallel by (i)~cutting $\tri$ open along $S_i$,
        and then (ii)~using the Jaco-Tollefson algorithm
        \cite[Algorithm~9.7]{jaco95-algorithms-decomposition}
        to test whether one of the resulting components is the product
        space $(\textrm{Torus}) \times [0,1]$.  If all incompressible tori are
        found to be boundary parallel then $K$ is small,
        and otherwise $K$ is large.
    \end{enumerate}
\end{algorithm}

\noindent
Regarding the individual steps:
\begin{itemize}
    \item Step~\ref{en-test-ideal} requires us to triangulate the
    complement of $K$.  Hass et~al.\ \cite{hass99-knotnp}
    show how to build a compact triangulation
    (with boundary triangles).  To make this an ideal
    triangulation we cone over the boundary, and retriangulate
    to remove internal (non-ideal) vertices.

    \item Step~\ref{en-test-enumerate} requires us to enumerate all
    extremal rays of $Q_0(\tri)$.  This is an expensive procedure
    (which is unavoidable, since there is a worst-case exponential
    number of extremal rays).
    For this we use the recent state-of-the-art tree traversal method
    \cite{burton10-tree}, which is tailored to the constraints and
    pathologies of normal surface theory and is found to be highly
    effective for larger problems.
    The tree traversal method works in the
    larger cone $Q(\tri)$, but it is a simple matter to insert the
    two additional linear equations corresponding to $\nu_x=0$.

    We also note that it is simple to identify the unique
    closed two-sided normal surface for which $Q(S)$ lies on the
    extremal ray $\mathbf{e}$.  Specifically, $Q(S)$ is either
    the smallest integer vector on $\mathbf{e}$ or, if that vector
    yields a one-sided surface, then its double.

    \item If we do not reach a conclusive result in
    step~\ref{en-test-essential}, then step~\ref{en-test-special} requires
    us to run the Jaco-Tollefson algorithm to test whether any
    incompressible torus is boundary-parallel.
    This algorithm is expensive: it requires us to work in a larger
    normal coordinate system, solve difficult enumeration
    problems, and perform intricate geometric operations.

    However, it is rare that we should reach this situation, and indeed
    for all $2979$ knots that we consider in Section~\ref{s-results},
    this scenario never occurs.  For some knots (e.g., satellite
    knots) it cannot be avoided, but there are additional fast methods
    for avoiding the Jaco-Tollefson algorithm even in these settings,
    which we describe in the full version of this paper.
\end{itemize}

\begin{theorem} \label{thm-large}
    Algorithm~\ref{alg-large} terminates, and its output is correct.
\end{theorem}

\begin{proof}
    It is clear that the algorithm terminates (there is no
    looping), and the correctness follows immediately from
    Theorems~\ref{thm:some extremal is closed essential} and
    \ref{thm-incompressible}.
    For details see Appendix~\ref{app-large}.
\end{proof}

% XTODO: Explain what happens in the more general setting: (i)~separating
% vs non-separating surfaces; (ii)~multiple boundary components.

\section{Computational results} \label{s-results}

Here we describe the results of running the algorithms of
Section~\ref{sec:Detecting compressing discs} over
significant collections of input knots.
These computational results emphasise that the new largeness testing
algorithm is both feasible to implement, and fast enough to be practical
for non-trivial inputs---both features that distinguish it from many of
its peers in algorithmic low-dimensional topology.

The algorithms were implemented in
C\nolinebreak\hspace{-.05em}\raisebox{.4ex}{\tiny\bf +}%
 \nolinebreak\hspace{-.10em}\raisebox{.4ex}{\tiny\bf +}
using the software package {\regina} \cite{burton04-regina,regina}.
The code is available from
\url{http://www.maths.uq.edu.au/~bab/code/}, and works with the
forthcoming {\regina} version~4.94.
Supporting data for the computations described here, including
triangulations of the knot complements and the corresponding
lists of admissible extreme rays of $Q_0(\tri)$,
can be downloaded from this same location.

All running times reported here are measured on a single
2.93~GHz Intel Core~i7 CPU.

\subsection{The census of knots up to 12 crossings}

Our first data set is the census of all $2977$ non-trivial prime knots
that can be represented with $\leq 12$ crossings.
Ideal triangulations of the knot complements were extracted from
the {\snappy} census tables \cite{snappy}, and then further
simplified where possible using {\regina}'s greedy heuristics
\cite{burton12-regina} to yield a final set of input triangulations
ranging from 2--26 tetrahedra in size.

The algorithms ran successfully over all $2977$ triangulations,
yielding the following results:

\begin{theorem}
    Of the $2977$ distinct non-trivial prime knots with up to $12$ crossings,
    $1019$ are large and $1958$ are small.
\end{theorem}

A full list of all $1019$ large knots
can be found in Appendix~\ref{app-knots}.
Regarding performance:
\begin{itemize}
    \item
    The enumeration of extremal rays of $Q_0(\tri)$ was extremely fast,
    with a maximum time of $47$~seconds, and a median time of just
    $0.08$~seconds.  This is a clear illustration of the benefits we
    obtain from Theorem~\ref{thm:some extremal is closed essential},
    which allows us to work in the restricted cone $Q_0(\tri)$ instead
    of the (typically much larger) cone $Q(\tri)$.

    The number of extremal rays of $Q_0(\tri)$ ranged from
    $0$ (for the figure~eight knot complement) up to $509$
    (for one of the 26-tetrahedron triangulations), with a median of $33$.

    \item
    Testing whether each candidate surface was essential was also
    extremely fast in most (but not all) cases.
    For each knot complement, we can sum the times required to process
    all candidate surfaces: the median time over all $2977$ knots was
    $\sim3.6$ seconds, and all but three of the knots had a processing time
    of under $12$~minutes.

    The remaining three knots, however, were significantly slower to
    process.  One required $\sim 3.9$~hours, one required $\sim 12.2$~hours,
    and one (the knot $12\mathrm{a}_{0779}$) was still running after $6$~days.
    However, in a striking illustration of how the algorithms depend
    strongly upon the underlying triangulations, when the code was
    run with a different random seed (which affected the simplification
    heuristics, and hence the triangulations obtained after slicing along
    surfaces), this worst-case knot $12\mathrm{a}_{0779}$ was fully processed
    in under $4$~minutes.
\end{itemize}

\subsection{The dodecahedral knots}

We now turn our attention to the dodecahedral knots
$D_f$ and $D_s$ as described by Aitchison and Rubinstein
\cite{aitchison92-cubings}.  These two knots exhibit remarkable
properties \cite{aitchison97-geodesic,neumann92-arithmetic},
and each can be represented with 20 crossings \cite{aitchison97-geodesic}.
Running our algorithms over them yields the following results:

\begin{theorem}
    The two dodecahedral knots $D_f$ and $D_s$ are both large.
    In particular, their complements contain
    closed essential surfaces of genus $3$.
\end{theorem}

We work with ideal triangulations of the knots $D_f$ and $D_s$ with
$46$ and $47$ tetrahedra respectively, which were kindly provided by Craig Hodgson.  
These are significantly larger
than the knots from the 12-crossing census; indeed, triangulations of
this size are typically considered well outside the range of feasibility
for normal normal surface theory.  Happily our algorithms now prove otherwise:
\begin{itemize}
    \item This time the enumeration of extremal rays of $Q_0(\tri)$
    was the bottleneck: for $D_f$ and $D_s$ this enumeration took
    roughly $2.8$ and $2.4$ days respectively.
    The number of admissible extremal rays was $72272$ and $73609$
    respectively.

    \item To test whether candidate surfaces were essential,
    the knot $D_s$ was completely processed in under $3$~minutes;
    in contrast, $D_f$ required roughly $4.4$~hours.
    Once again, we see that this part of the algorithm depends heavily
    upon the underlying triangulation: when running with a different random
    seed, $D_f$ was likewise processed in just a few minutes.
\end{itemize}

%%%%%%%%%%%%%%%%%%%%%%%%%%%%%%%%%%%%%%%%%%%%%%%%%%%%%%%%%%%%%%%%%%%%%%%%
%
%   Bibliography
%
%%%%%%%%%%%%%%%%%%%%%%%%%%%%%%%%%%%%%%%%%%%%%%%%%%%%%%%%%%%%%%%%%%%%%%%%

\bibliographystyle{amsplain}
\bibliography{pure}

\bigskip

%%%%%%%%%%%%%%%%%%%%%%%%%%%%%%%%%%%%%%%%%%%%%%%%%%%%%%%%%%%%%%%%%%%%%%%%
%
%   Appendix
%
%%%%%%%%%%%%%%%%%%%%%%%%%%%%%%%%%%%%%%%%%%%%%%%%%%%%%%%%%%%%%%%%%%%%%%%%
\newpage

\appendix

\section*{Appendix: Additional proofs, examples and data}

This paper is concerned with finding special surfaces in 3--manifolds. To keep the paper short, we mainly considered the case of finding a closed essential surface in a knot exterior in $S^3,$ and have introduced the bare minimum of the new theory and algorithms required for this application. The most natural generalisation of our results and algorithms is to the setting of compact 3--manifolds with boundary consisting of a finite union of tori. These manifolds are related to the Geometrisation Theorem of Thurston and Perelman, as they are key building blocks in the so-called JSJ decomposition of a closed, irreducible, orientable 3--manifold into geometric pieces.

\section{General definitions}

In this first appendix, we give more technical definitions of knot and link manifolds, essential surfaces, triangulations and normal surfaces.

\subsection{Knots and 3--manifolds}

Suppose $M$ is an orientable 3--manifold (possibly with boundary). There
are two key properties that are often required of $M$ in geometric
topology. The first is that every embedded 2--sphere in $M$ bounds a ball to at
least one side; in this case $M$ is \emph{irreducible}. The second is
that for each boundary component $B$ of $M,$ the inclusion homomorphism
$\pi_1(B)\to \pi_1(M)$ is injective; in this case $M$ is
$\partial$--irreducible. (A geometric interpretation of this algebraic
property is given in the next section.) For instance, the 3--sphere
$S^3$ is irreducible, and if $K$ is an embedding of $S^1$ in $S^3,$
called a \emph{knot}, then the \emph{knot exterior} $M= S^3\setminus
N(K),$ where $N(K)$ is a small open neighbourhood of $K,$ is also
irreducible, and it is $\partial$--irreducible if and only if $K$ is
non-trivial.
The \emph{knot complement} $S^3\setminus K$ is homeomorphic
to the interior of the knot exterior $S^3\setminus N(K),$ and it is
sometimes useful to switch between the complement (a non-compact
manifold) and the exterior (a manifold-with-boundary).

\subsection{Surfaces in 3-manifolds}
\label{subsec:essential}

The following definition of an essential surface, along with an extensive discussion of their properties, can be found in Shalen \cite{shalen02-representations}, \S1.5.

\begin{defn}[Essential surface]
A properly embedded surface $S$ in the compact, irreducible, orientable 3--manifold $M$ is \emph{essential} if it has the following properties:
\begin{enumerate}
\item $S$ is bicollared;
\item the inclusion homomorphism $\pi_1(S_i)\to \pi_1(M)$ is injective for each component $S_i$ of $S$;
\item no component of $S$ is a 2--sphere;
\item no component of $S$ is boundary parallel; and
\item $S$ is non-empty.
\end{enumerate}
\end{defn}

Of interest to this paper is the following geometric interpretation of the second property, see Shalen \cite{shalen02-representations} for more details. A \emph{compression disc} for the surface $S$ is a disc $D\subset M$ such that $D \cap S = \partial D$ and $\partial D$ is homotopically non-trivial in $S$ (i.e. does not bound a disc on $S$). In particular, if $S$ has a compression disc, then $\pi_1(S_i)\to \pi_1(M)$ is not injective for some component of $S.$ It follows from classical work of Papakyriakopoulos that the converse is also true. Detecting compression discs is the topic of Section~\ref{sec:Detecting compressing discs}.

If the surface $S$ has non-empty boundary and is properly embedded in a 3--manifold with boundary, there is an additional requirement that one has in order for $S$ to be topologically significant; namely that it be \emph{$\partial$--incompressible}. In the case of interest for this paper, where all boundary components of $M$ are tori, it turns out that incompressible implies $\partial$--incompressible. This is one reason why our algorithms most naturally generalise to the class of link complements.

\subsection{Triangulations}
\label{subsec:Triangulations}

The notation of \cite{jaco03-0-efficiency} and \cite{tillmann08-finite} will be used in this paper. A \emph{triangulation}, $\tri,$ of a compact 3--manifold $M$ consists of a union of pairwise disjoint 3--simplices, $\widetilde{\Delta} = \{ \sigma_1, \ldots, \sigma_t\},$ a set of face pairings, $\Phi,$ and a natural quotient map $p\co \widetilde{\Delta} \to \widetilde{\Delta} / \Phi = M.$ Since the quotient map is injective on the interior of each 3--simplex, we will refer to the image of a 3--simplex in $M$ as a \emph{tetrahedron} and to \emph{its} faces, edges and vertices with respect to the pre-image. Similarly for images of 2--, 1-- and 0--simplices, which will be referred to as \emph{faces}, \emph{edges} and \emph{vertices} in $M$ respectively. 
%For edge $e,$ the number of pairwise distinct 1--simplices in $p^{-1}(e)$ is termed its \emph{degree}, denoted $d(e).$ 
If an edge is contained in $\partial M,$ then it is termed a \emph{boundary edge}; otherwise it is an \emph{interior edge}.

If $M$ is the interior of a compact manifold with non-empty boundary, an \emph{ideal triangulation}, $\tri,$ of $M$ consists of a union of pairwise disjoint 3--simplices, $\widetilde{\Delta},$ a set of face pairings, $\Phi,$ and a natural quotient map $p\co \widetilde{\Delta} \to \widetilde{\Delta} / \Phi = P,$ such that $M = P \setminus P^{(0)}$ is the complement of the 0--skeleton. The quotient space $P$ is usually called a \emph{pseudo-manifold}.

For brevity, we will refer to a 3--manifold $M$ imbued with a (possibly ideal) triangulation $\tri=(\widetilde{\Delta}, \Phi, p)$ as a
\emph{triangulated 3--manifold}. Throughout, we will assume that $M$ is \emph{oriented}, that all tetrahedra in $M$ are oriented coherently and the tetrahedra in $\widetilde{\Delta}$ are given the induced orientation.

\subsection{Normal surfaces}

A \emph{normal corner} is an interior point of a 1--simplex. A \emph{normal arc} is a properly embedded straight line segment on a 2--simplex with boundary consisting of normal corners. A \emph{normal disc} is a properly embedded disc in a 3--simplex whose boundary consists of normal arcs no two of which are contained on the same face of the 3--simplex; moreover, the normal disc is the cone over its boundary with cone point the barycentre of its normal corners. It follows that the boundary of a normal disc consists of either three or four normal arcs, and it is accordingly called a \emph{normal triangle} or a \emph{normal quadrilateral}. Moreover, a normal disc is uniquely determined by its intersection with the 1--skeleton.

A \emph{normal isotopy} is an isotopy of $M$ that leaves all simplices invariant. Up to normal isotopy, there are 7 types of normal discs in each tetrahedron. We denote $\square$ the set of all isotopy classes of normal quadrilaterals in $\widetilde{\Delta},$ and identify this with the set of all isotopy classes of normal quadrilaterals in $M$ via $p.$ Following Haken, we will connect topology to linear programming via certain functions $x\co \square \to \R^{3t}.$

A \emph{normal surface} $S$ in a triangulated 3--manifold is a properly embedded surface which intersects each 3--simplex in a union of pairwise disjoint normal discs; such a surface is often termed a \emph{spun-normal} surface if one of its connected components contains infinitely many normal discs. 

\subsection{Standard coordinates}

In classical normal surface theory one only considers normal surfaces where each tetrahedron contains at most finitely many disjoint normal disks. The number of parallel copies each type of normal disk specifies the normal surface, and in this way normal surfaces may be studied via $7t$-dimensional vectors with positive integer coordinates, where $t$ is the number of tetrahedra in the triangulation under consideration. These surfaces can be described by a set of linear equations arising from the fact that the normal discs in one tetrahedron have to match up with normal discs in each adjacent tetrahedron across their common face. If $\Delta$ denotes the set of all normal isotopy classes of normal triangles, then the natural coordinate projection $\R^{\Delta \cup \square} \to \R^\square$ takes the standard coordinate of a normal surface to its $Q$--coordinate.

\begin{figure}[h!]
\centering
\includegraphics[width=0.15\textwidth]{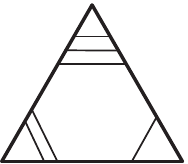}
\caption{The intersection of a normal surface with a typical face of
the 2-skeleton.} \label{3arcst}
\end{figure}

Consider a face of the 2-skeleton of
the triangulation of $M$. Any normal surface in $M$ must intersect
this triangle in \emph{normal arcs}, that is arcs which start and
end on different edges.
There are three types of normal arc in any face of the 2-skeleton,
as shown in Figure \ref{3arcst}. Each of these represents an edge of
a normal disc on each side. These normal discs must match up, as
shown in Figure \ref{matchup}. Note that on each side of a triangle
of the 2-skeleton, there are only two possible types of normal disc,
one triangle and one quadrilateral, which can give rise to each type of
normal arc.

\begin{figure}[h!]
\centering
\includegraphics[width=0.4\textwidth]{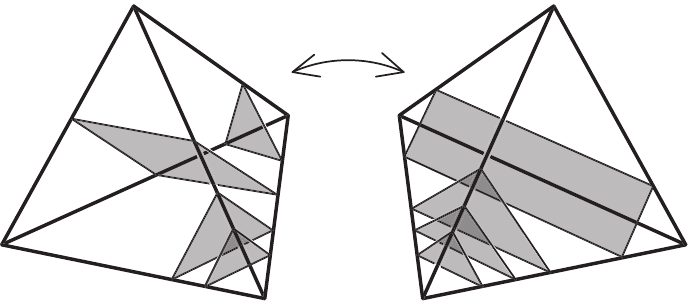}
\caption{The normal discs must match up on the common face of two
adjacent tetrahedra.} \label{matchup}
\end{figure}

Let $t_1,$ $q_1$ be the normal triangle and quadrilateral types in one tetrahedron, sharing the same normal arc type with the normal triangle and quadrilateral types $t_2$ and $q_2$ in an adjacent tetrahedron. We then obtain the linear equation
$$
t_1 + q_1 = t_2 + q_2,
$$
with three such equations for each internal face of the triangulation. The space of all non-negative solutions to these equations in $\R^{\Delta \cup \square}$ is the usual normal surface cone.

%\section{An explicit form of the $Q$--matching equations}
%\label{app:Q-eqns}
%
%We denote the three quadrilateral disc types supported by tetrahedron $\sigma_k \in \widetilde{\Delta}$ by $q_k, q'_k, q''_k$ such that the ordering agrees with a right-handed screw orientation, see Figure~\ref{fig:quad coods}.
%Let $a^{(k)}_{ij}$ be the number of edges in $p^{-1}(e_j)\cap \sigma_i$ which are \emph{disjoint} from $q_i^{(k)}.$
%Then the convention of signs and order of discs implies that the Q--matching equation of $e_j$ is:
%\begin{align*}
%    0 &= \sum_{i=1}^n a_{ij} \  (x(q'_i)-x(q''_i)) + a'_{ij} \ ( x(q''_i)-x(q_i)) +    a''_{ij} \ ( x(q_i) - x(q'_i)) \\  
%    &= \sum_{i=1}^n (a''_{ij} - a'_{ij})\ x(q_i) + (a_{ij} - a''_{ij})\ x(q'_i) + (a'_{ij} - a_{ij})\ x(q''_i).
%\end{align*}

\section{Example: The figure-eight knot}
\label{app:fig8}

\begin{figure}[h]
  \begin{center}
      \subfigure[The figure 8 knot spans a Klein bottle]{
      \includegraphics[scale=0.6]{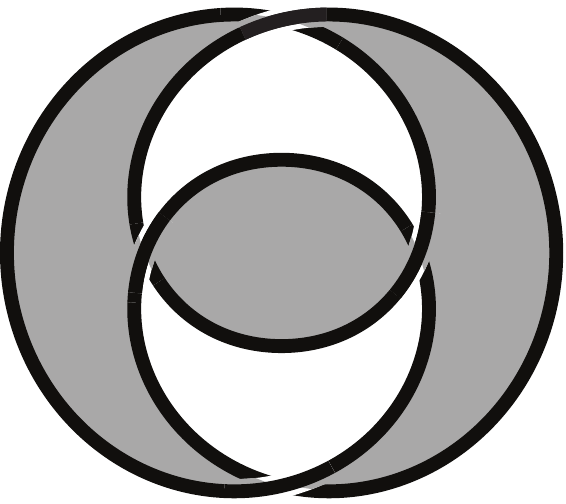}}
      \qquad\qquad
    \subfigure[Thurston's ideal triangulation of the complement]{
      \includegraphics[scale=0.9]{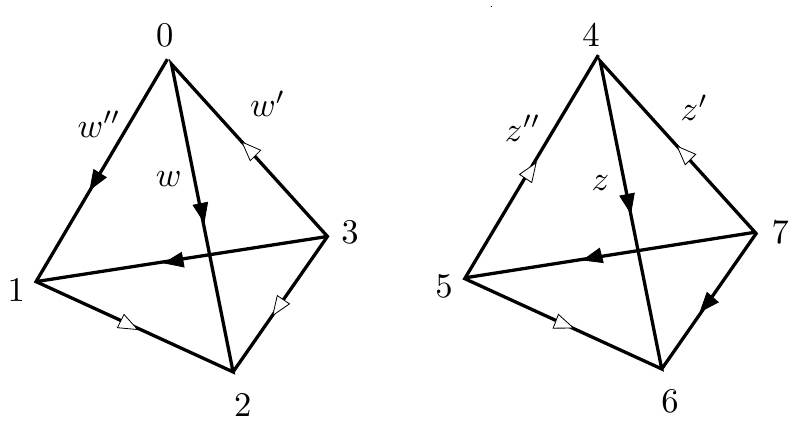}}
  \end{center}
 \caption{The figure eight knot}
    \label{fig:fig8_tets}
\end{figure}

Let $M$ denote the complement of the figure eight knot, which is shown in Figure \ref{fig:fig8_tets}(a). An oriented, ideal triangulation of $M$ is encoded in Figure \ref{fig:fig8_tets}(b). Since $M$ is oriented, we may compute the $Q$--matching equations from the figure. To simplify notation, we use the dual labelling scheme of \cite{tillmann08-finite}, and denote the quadrilateral types dual to the edges labelled $w^{(k)}$ and $z^{(k)}$ by $p^{(k)}$ and $q^{(k)}$ respectively. The $Q$--matching equations for the two edges are equivalent, and one has:
\begin{equation*}
0 = p + p' - 2 p''+q + q' - 2 q''.
\end{equation*}
This implies that the cone $Q(\tri )$ is five--dimensional. 
A direct calculation reveals that $Q (\tri )$ has four admissible extremal  rays; all have minimal representative a once--punctured Klein bottle; such a Klein bottle is shown in Figure \ref{fig:fig8_tets}(a). 
Their normal $Q$--coordinates and boundary slopes are listed in Table \ref{tab:surfaces fig8}. This calculation in particular shows that no spun-normal surface is a Seifert surface for the knot. 

\begin{table}[h]
\begin{center}
\begin{tabular}{ c | r | r | r}
solution & $\nu (\mu )$  & $\nu (\lambda )$ & slope \\
\hline
(2,0,0,0,0,1) & 1 &4  &        --4  \\
(0,2,0,0,0,1) & --1 & 4  &       4   \\
(0,0,1,2,0,0) & --1 & --4  & --4 \\
(0,0,1,0,2,0) &  1&  --4  &     4 \\
\end{tabular}
\end{center}
\caption{Normal surface in the figure eight knot complement}
\label{tab:surfaces fig8}
\end{table}

The induced triangulation of a vertex linking torus, $T,$ is shown in Figure \ref{fig:torus triang fig8}. 
The dual labelling allows us to read off the linear functional $\nu(\gamma)$ for the path $\gamma$ that is transverse to the 1--skeleton on $T.$ If the path $\gamma$ on $T$ exits a triangle across the edge opposite the vertex with label $w^{(k)},$ then in $\nu(\gamma),$ we record $-p^{(k)},$ and if it enters a triangle across the edge opposite the vertex labelled $w^{(k)},$ then in $\nu(\gamma),$ we record $+p^{(k)}.$ Similarly for the labels $z^{(k)}.$

\begin{figure}[h!]
\begin{center}
  \includegraphics{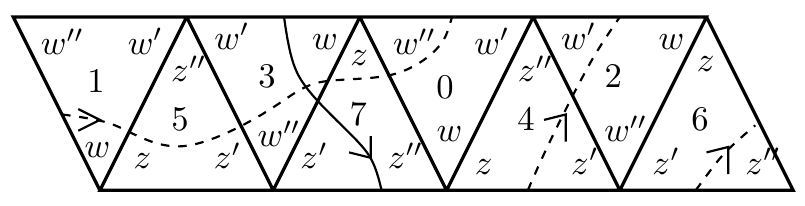}
\end{center}
    \caption{The induced triangulation of the vertex linking torus, where the sides of the rectangle are identified by translations parallel to its sides and triangle $i$ is dual to vertex $i$ in Figure~\ref{fig:fig8_tets}. The shown elementary curves are the standard meridian $\mu$ (solid) and longitude $\lambda$ (dashed).}
    \label{fig:torus triang fig8}
\end{figure}

In this fashion (and using the $Q$--matching equation to simplify the expressions), we determine the linear functionals associated to the standard peripheral curves:
\begin{align*}
\nu (\lambda ) &=  2 p +2 p' - 4 p'',\\
\nu (\mu ) &= -p'+p''-q+q''.
\end{align*}
One can now verify that $Q_0(\tri )=\{0\}.$ Whence any closed embedded normal surface is a union of vertex linking tori. 

\section{Proof of Theorem~\ref{thm:some extremal is closed essential}}
\label{app:Jaco-Oertel}

In this appendix, we prove Theorem~\ref{thm:some extremal is closed essential} in the more general setting of a 3--manifold with boundary consisting of a union of tori. The definition of the boundary map and of $Q_0(\tri)$ is generalised to the case of multiple boundary components as follows. For each vertex linking torus $T_k,$ we obtain a well-defined homomorphism $\nu_{k,x}\co H_1(T_k; \R) \to \R,$ and we define $\nu_x = \oplus_k\; \nu_{k,x},$ where the sum is taken over all ideal vertices. The surface in Theorem~\ref{thm:admissible integer solution gives normal} is closed if and only if $\nu_x =0$ (see \cite{tillmann08-finite}, Proposition 3.3). We then define $Q_0(\tri) = Q(\tri) \cap \{ x \mid \nu_x=0\},$  and call a 2--sided, connected normal, surface $F$ with $Q(F)$ on an extremal ray of $Q_0(\tri)$ a \emph{$Q_0$--vertex surface.}

The weight of the normal surface $F$ is the cardinality of its intersection with the 1--skeleton,  $\wt(F) = |F \cap M^{(1)}|.$ If $F$ is closed, then its weight is finite.

\begin{figure}[h]
\begin{center}
  \subfigure[]{
      \includegraphics[height=5cm]{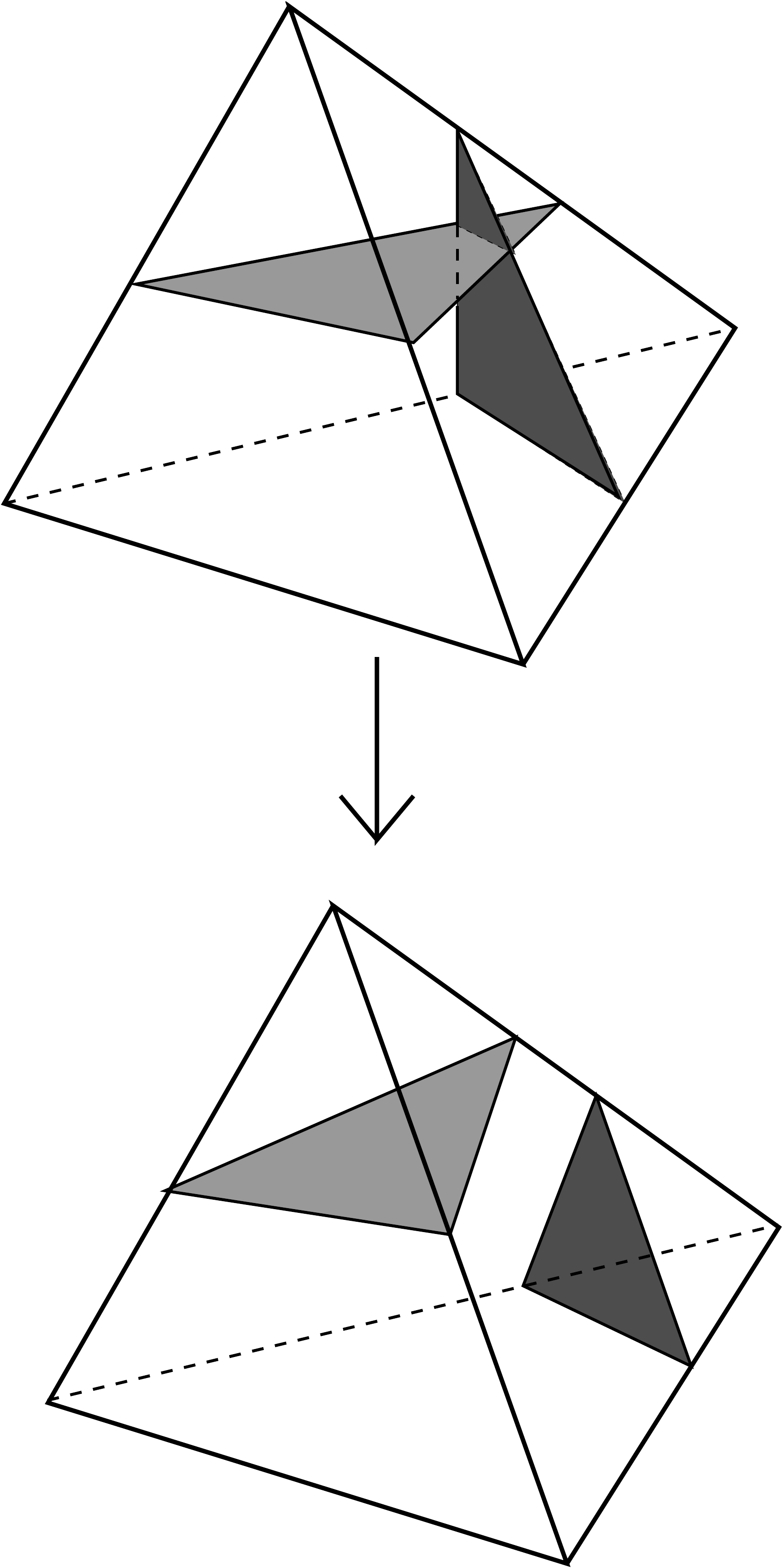}
    } 
    \quad
\subfigure[]{
      \includegraphics[height=5cm]{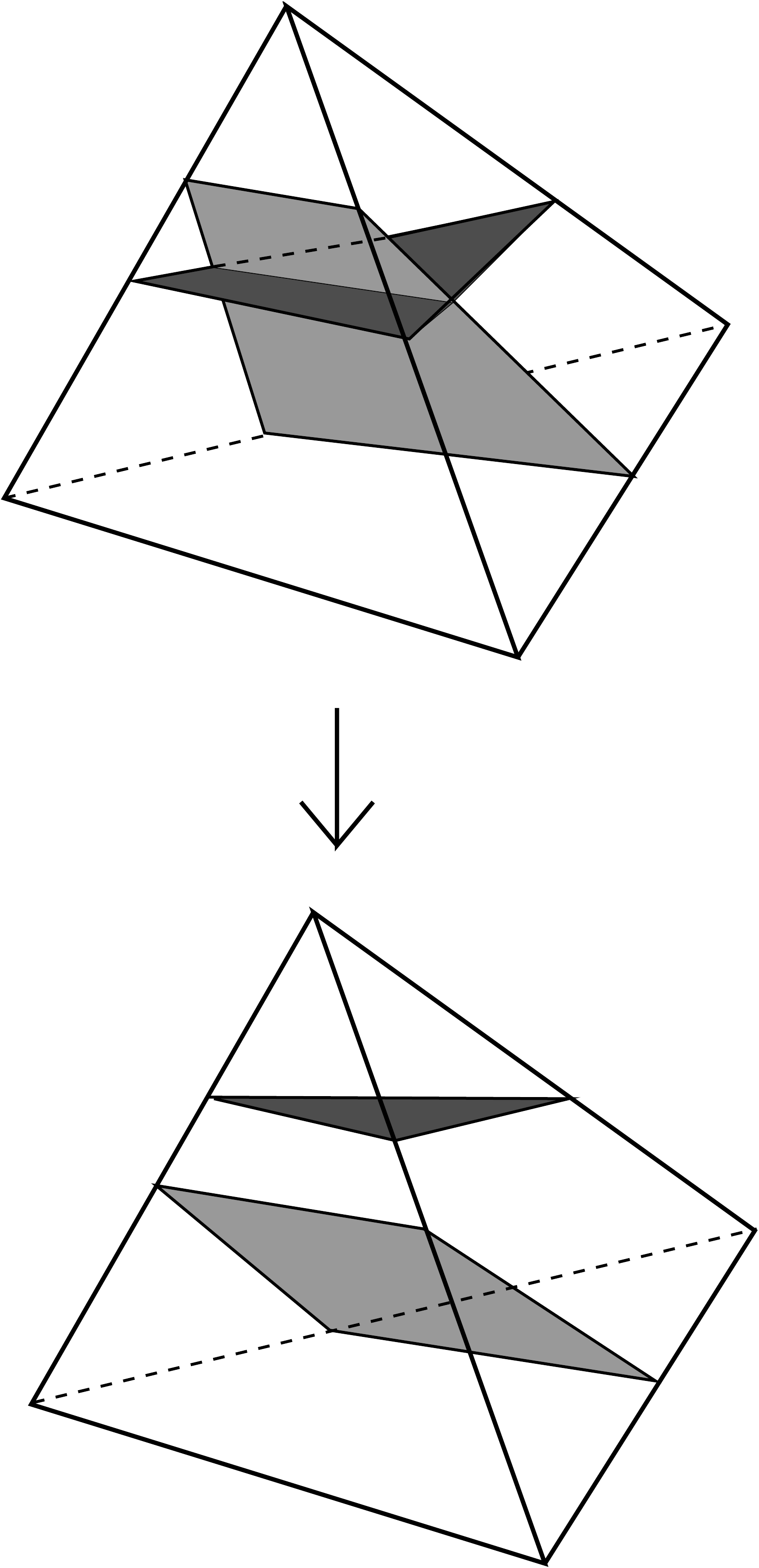}
    } 
    \quad
\subfigure[]{
      \includegraphics[height=5cm]{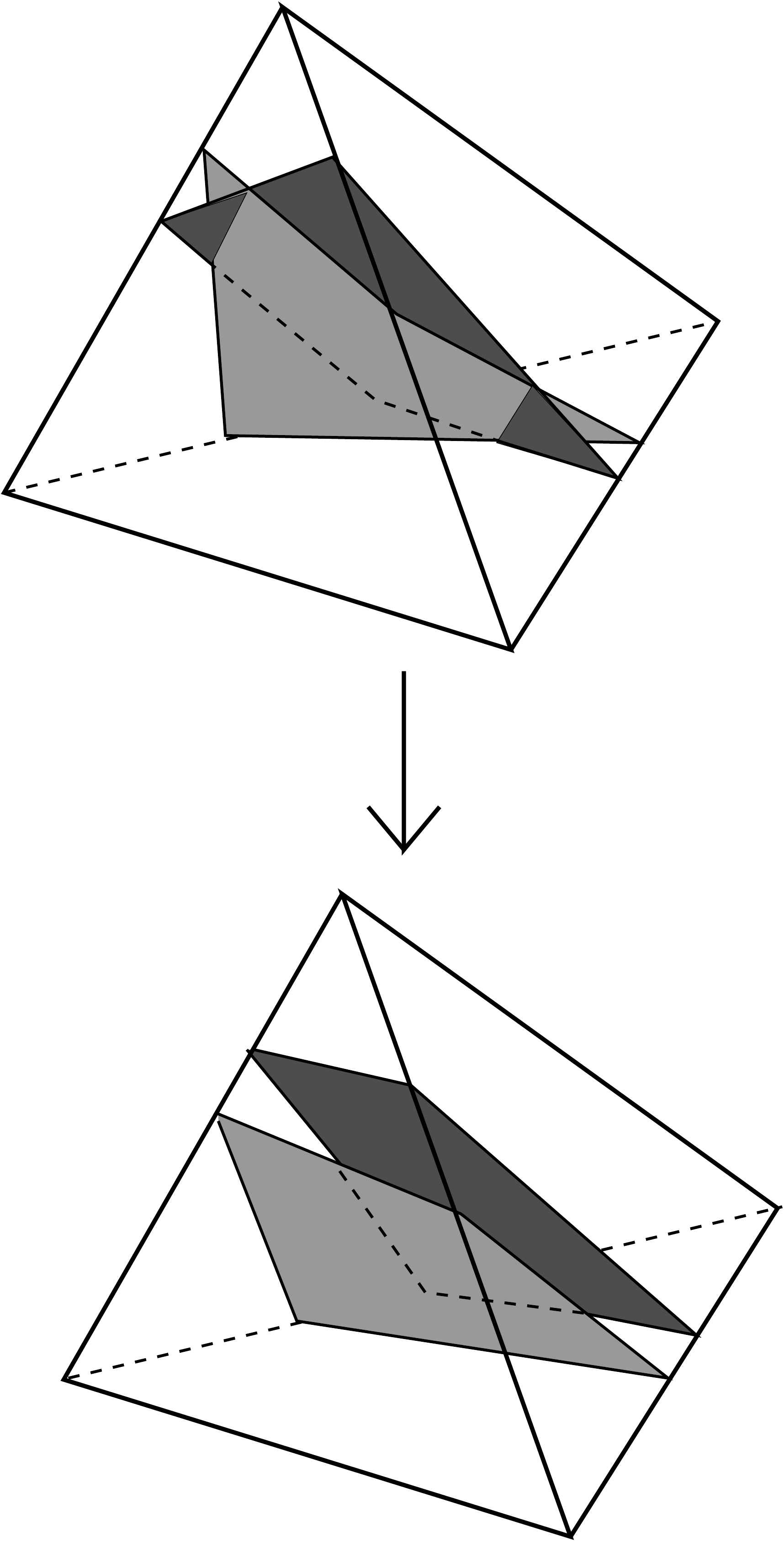}
    } 
    \quad
\subfigure[]{
      \includegraphics[height=5cm]{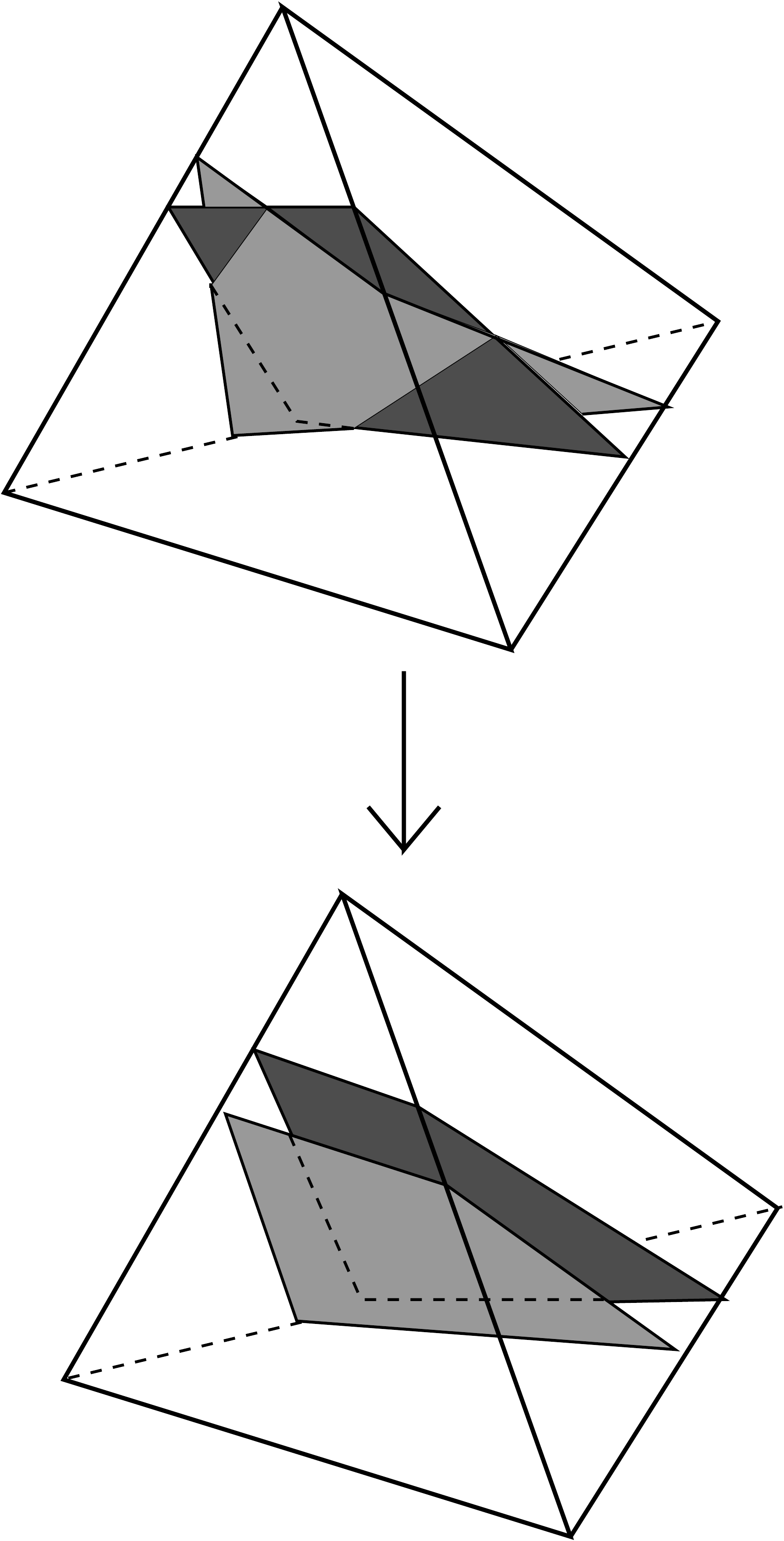}
    } 
\end{center}
    \caption{Regular exchange of normal discs}
     \label{fig:regular exchange of discs}
\end{figure}

Two normal surfaces are \emph{compatible} if they do not meet a tetrahedron in quadrilateral discs of different types. In this case, the sum of their normal coordinates is the coordinate of a normal surface.
Suppose $F_1$ and $F_2$ are normal surfaces that are compatible, not vertex linking surfaces, and in general position. 
Then $x(F_1)+x(F_2)$ is an admissible solution to the $Q$--matching equations, and hence represented by a unique normal surface without vertex linking components; denote this surface $F.$ The surface $F$ is obtained geometrically as follows. At each component of $F_1 \cap F_2,$ there is a natural choice of \emph{regular switch} between normal discs, such that the result is again a normal surface. See Figure~\ref{fig:regular exchange of discs} for some possible configurations involving only two discs. Denote $N(F_1 \cap F_2)$ a small, open, tubular neighbourhood of $F_1 \cap F_2.$
The connected components of $(F_1 \cup F_2) \setminus N(F_1 \cap F_2)$ are termed \emph{patches}.

Deleting any vertex linking tori that arise gives the surface $F,$ and we write $F + \Sigma = F_1 + F_2,$ where $\Sigma$ is (possibly empty or a possibly infinite) union of vertex linking tori. This is called the \emph{Haken sum} of $F_1$ and $F_2.$
Both weight and Euler characteristic are additive under this sum. So we have
\begin{align*}
\wt(F_1) + \wt(F_2) &=\wt(F) + \wt(\Sigma),\\
\chi(F_1) + \chi(F_2) &=\chi(F),
\end{align*}
since $\chi(\Sigma)=0.$

The sum $F+\Sigma=F_1 + F_2$ is said to be in \emph{reduced form} if there is no Haken sum $F+\Sigma'=F'_1 + F'_2,$ where $F'_i$ is isotopic to $F_i$ in $M,$
 $F'_1 \cap F'_2$ has fewer components than $F_1 \cap F_2$ and $\Sigma'$ is a union of vertex linking tori.
  It should
be noted that in these two sums, the embedding of $F$ in $M$ is the same
(these are not equalities up to isotopy), and that any sum can be
changed to a sum in reduced form.

\begin{reptheorem}{thm:some extremal is closed essential}
Suppose $M$ is the interior of a compact, irreducible and $\partial$--irreducible manifold with boundary consisting of a union of tori, and let $\tri$ be an ideal triangulation of $M.$ If $M$ contains a closed, essential surface $S,$ then there is a normal, closed essential surface $F$ with the property that $x(F)$ lies on an extremal ray of $Q_0(\tri).$ Moreover, if $\chi(S)<0,$ then there is such $F$ with $\chi(F)<0.$ 
\end{reptheorem}

\begin{proof}
Suppose $M$ contains a closed, essential surface. It follows from a standard argument (see, for instance, \cite{jaco84-haken} and \cite{jaco03-0-efficiency}) that there is a closed, essential, normal surface $S$ in $M.$ It remains to show that $S$ may be chosen such that $S$ is a $Q_0$--vertex surface. Replace $S$ by a normal surface that has least weight amongst all normal surfaces isotopic (but not necessarily normally isotopic) to $S.$ Denote this least weight surface $S$ again. 

Suppose $S$ is not a $Q_0$--vertex surface. Then $$n x(S) = \sum n_i x(V_i),$$ where $n, n_i \in \mathbb{N}$ and either $V_i$ or $2V_i$ is a $Q_0$--vertex surface for each $i.$ The two cases arise from the fact that we require a $Q_0$--vertex surface to be 2--sided and connected: If $V$ corresponds to the first integer lattice point on an admissible extremal ray of $Q_0(\tri)$ and $V$ is 1--sided, then the corresponding $Q_0$--vertex surface is $2V,$ obtained by taking the boundary of a regular neighbourhood of $V.$

We denote $nS$ the normal surface obtained by taking $n$ parallel copies of $S.$ Clearly, $x(nS) = nx(S),$ and since $S$ has least weight normal surface in its isotopy class, so does $nS$ because $S$ is 2--sided. To sum up, $nS$ is a closed, essential, normal surface which has least weight amongst all normal surfaces in its isotopy class.

For any $i,$ either $V_i$ or $2V_i$ is a $Q_0$--vertex surface. In the first case, we can write
$$nS + \Sigma = V + W,$$
where  $V = V_i$ and $x(W) =  -x(V_i) + \sum n_i x(V_i).$ Now $nS$ is 2--sided and of least weight, so Kang~\cite{kang05-spun}, Theorem 5.4 (which is an adaptation of the proof of \cite{jaco84-haken}, Theorem 2.2, to this context), shows that if one writes $nS + \Sigma = V' + W'$ in reduced form, then $V'$ is incompressible. Now $V'$ is isotopic in $M$ to $V,$ and hence $V$ is also incompressible. If $2V_i$ is a $Q_0$--vertex surface, then we apply the above argument to $2nS,$ writing $2nS + \Sigma = V + W,$ where $x(V)=2x(V_i)$ and
$x(W) =  -x(V) + 2\sum n_i x(V_i).$ In either case, we obtain an incompressible $Q_0$--vertex surface $V.$ Since Euler characteristic is additive and $S$ is not a sphere, there is some $V_i$ with $\chi(V_i)\le 0.$ If $\chi(V_i)<0,$ then $V$ is essential, and if $\chi(V_i)=0,$ then $V$ may be essential or boundary-parallel. Hence if $\chi(S)<0,$ the proof of the theorem is complete.

Hence assume $\chi(S)=0,$ and finish the proof with an argument from the proof of Proposition 6.3.21 from \cite{matveev03-algms}. For the sake of a contradiction, suppose that some $V_i$ is a boundary parallel torus that is not vertex linking. Write  $nS + \Sigma = V + W,$ where $V$ is a boundary parallel torus that is not vertex linking, and suppose that the sum is in reduced form. As in \cite{jaco84-haken}, Lemma 2.1, it follows that each patch is incompressible and not a disc. Denote $M_2$ a component of $M \setminus V$ that is homeomorphic to $V \times (0,1).$ If $W \cap M_2\neq \emptyset,$ then it consists of a pairwise disjoint union of annuli. Choosing an innermost annulus $A \subseteq W \cap M_2,$ there is an annulus $A' \subseteq V$ such that $\partial A' = \partial A,$ and there is an isotopy from $A$ to $A'$ keeping the boundaries fixed. But this implies that $V + W$ is not in reduced form, giving a contradiction.
\end{proof}

\begin{remark}
The statement of Theorem~\ref{thm:some extremal is closed essential} considerably strengthens the statement of Theorem~5.5 in \cite{kang05-spun}, and our proof fills a gap in its proof.
\end{remark}

\section{Proof of Theorem~\ref{thm-incompressible}} \label{app-incompressible}

Here we give a full proof of correctness for
Algorithm~\ref{alg-incompressible}.  Recall that this algorithm takes
a closed two-sided normal surface $S$ of genus $g$ within an ideal
triangulation $\tri$ of a knot complement in $S^3$, and tests whether
this surface is incompressible.

\begin{reptheorem}{thm-incompressible}
    Algorithm~\ref{alg-incompressible} terminates, and its output is correct.
\end{reptheorem}

\begin{proof}
    As noted in the main body of the paper, termination is
    straightforward: each time we loop back to step~\ref{en-alg-simplify}
    we have strictly fewer tetrahedra than before.  We now devote
    ourselves to proving the many claims that are made throughout the
    statement of Algorithm~\ref{alg-incompressible}.

    Before proceeding, however, we make a brief note regarding
    irreducibility.
    Since the underlying knot is embedded in $S^3$,
    every sphere in the complement $M$ must bound a ball; that is,
    $M$ is irreducible.  As a result, the two manifolds $M_1$ and $M_2$
    are likewise irreducible, with the following possible exception:
    it might be the case that $M_2$ has an embedded sphere that
    separates the genus $g$ boundary $B_2$ on one side from the
    torus boundary $B_v$ on the other.  However, in this case (by
    reattaching $M_1$) we find that $S$ is contained in a 3-ball within
    the complement $M$; in particular, $S$ has a compressing disc within
    $M_1$.

    We proceed now with proofs of the various claims made in
    Algorithm~\ref{alg-incompressible}.

    \begin{itemize}
        \item \emph{In step~\ref{en-alg-truncate} we claim that cutting along
        $S$ yields two compact manifolds}.

        This is because $S$ is a closed surface embedded in a knot
        complement, which itself is a submanifold of the 3-sphere.
        Since every closed surface in the 3-sphere is separating, the
        claim follows.

        \item \emph{In step~\ref{en-alg-nodisc} we claim that, if the surface
        $E$ cannot be found in $\tri_1$ \emph{and} it cannot be found in
        $\tri_2$, then the original surface $S$ must be incompressible.}

        Suppose that $S$ were compressible, with a compression disc
        in some $M_i$.  If this $M_i$ is irreducible,
        then by a result of Jaco and Oertel \cite[Lemma~4.1]{jaco84-haken}
        there is a \emph{normal} compressing disc in $\tri_i$.
        Since the underlying knot is non-trivial, this compressing disc
        must meet the genus $g$ boundary $B_i$ (not the torus
        boundary $B_v$), and so it is a surface of the type we are
        searching for.
        If $M_i$ is reducible then (from earlier) we have $i=2$, there
        is a compressing disc within the irreducible manifold $M_1$, and
        by the argument above the surface $E$ can be found within $\tri_1$.

        \item \emph{In step~\ref{en-alg-crush} we claim that
        the new triangulation $\tri'_i$ has strictly fewer tetrahedra
        than $\tri_i$.}

        This is because $E$ is connected but not a vertex link, and
        therefore contains at least one normal quadrilateral.
        As noted in Section~\ref{subsec:Crushing},
        this means that at least one tetrahedron of $\tri_i$ is deleted
        in the Jaco-Rubinstein crushing process.

        \item \emph{In step~\ref{en-alg-crush} we claim that
        if $\tri'_i$ has a component with the same genus boundary
        (or boundaries) as $\tri_i$ then this component represents the same
        manifold $M_i$, and if not then $S$ is compressible.}

        Since the surface $E$ that we crush is connected with positive
        Euler characteristic and can be embedded within a knot complement
        (and hence the 3-sphere), it follows that $E$ is either a sphere
        or a disc.
        From Section~\ref{subsec:Crushing}, this means that when we crush $E$
        in the triangulation $\tri_i$, the resulting manifold is obtained
        from $M_i$ by zero or more of the following operations:
        \begin{itemize}
            \item undoing connected sums;
            \item cutting open along properly embedded discs;
            \item filling boundary spheres with 3-balls;
            \item deleting 3-ball, 3-sphere, $\R P^3$, $L_{3,1}$ or
            $S^2 \times S^1$ components.
        \end{itemize}

        We first note that, since all of the manifolds we consider can
        be expressed as submanifolds of $S^3$, we will never create or
        delete an $\R P^3$, $L_{3,1}$ or $S^2 \times S^1$ component.

        Suppose that $M_i$ is irreducible.  Then undoing a connected sum
        simply has the effect of creating an extra 3-sphere component.
        If we ever cut along a properly embedded disc that is \emph{not}
        a compressing disc, then likewise this just creates an extra
        3-ball component.  If we cut along a compressing disc, then
        this yields one or two pieces with strictly smaller total
        boundary genus than before; moreover, since the underlying knot
        is non-trivial, the first such compression must take place along
        the genus $g$ boundary $B_g$ (not $B_v$) and so $S$ must be
        compressible.  Together these observations establish the full
        claim above.

        If $M_i$ is not irreducible, then as noted above $i=2$, and
        $M_2$ must contain a sphere that separates the boundary component
        $B_2$ from $B_v$.  Here there is an extra complication:
        either undoing a connected sum or cutting along a non-compressing disc
        might have the effect of splitting the manifold into two
        components, one with the single boundary $B_2$ and the other
        with the single boundary $B_v$.  In this case there will no
        longer be a component in $\tri_2'$ with the same genus boundaries
        as $\tri_2$; however, as noted earlier $S$ is compressible on
        the side of $M_1$, and so the claim above remains correct.
        \qedhere
    \end{itemize}
\end{proof}

\section{Proof of Theorem~\ref{thm-large}} \label{app-large}

Our final proof is of the correctness of Algorithm~\ref{alg-large},
the full algorithm for testing whether a non-trivial knot $K$ in $S^3$ is
large.

\begin{reptheorem}{thm-large}
    Algorithm~\ref{alg-large} terminates, and its output is correct.
\end{reptheorem}

\begin{proof}
    As noted in the main body of the paper, it is clear that the
    algorithm terminates since there is no looping.  All that remains is
    to prove that its output is correct.

    Throughout this proof we
    implicitly use Theorem~\ref{thm-incompressible} to verify that
    all calls to Algorithm~\ref{alg-incompressible} are themselves correct.
    We note that, since the knot $K$ is non-trivial, the complement
    is irreducible and $\partial$--irreducible and so the conditions of
    Theorem~\ref{thm:some extremal is closed essential} apply.
    We also note that all closed surfaces embedded in a knot complement
    in $S^3$ must be orientable, and so we implicitly treat all closed surfaces
    as orientable from here onwards.

    From Theorem~\ref{thm:some extremal is closed essential},
    the knot $K$ is large if and only if one of the closed normal
    surfaces $S_i$ in our list (excluding spheres) is essential.
    We note that each genus $\geq 2$ surface in the list is essential
    if and only if it is incompressible (since such a surface cannot be
    boundary parallel),
    and each torus in the list is essential if and only if it is
    (i)~incompressible and (ii)~not boundary parallel.

    Steps~\ref{en-test-essential} and \ref{en-test-special} of
    Algorithm~\ref{alg-large} test precisely these conditions, and so
    the algorithm is correct.  The only reason for the complex logic in
    these steps is so that we can use Algorithm~\ref{alg-incompressible}
    exclusively if possible, and only fall through to the more expensive
    Jaco-Tollefson algorithm when absolutely necessary.
\end{proof}

% \newpage

\section{Tables of large knots} \label{app-knots}

We finish with the detailed results of running
Algorithm~\ref{alg-large} over all $2977$ knots in the census of
non-trivial prime knots with $\leq 12$ crossings.
The following list presents all $1019$ knots in this census
that are large.  The knots are identified using
their names in the {\knotinfo} database \cite{www-knotinfo}.

\newpage
{
\centering\footnotesize
\begin{multicols}{9}
$8_{16}$ \\
$8_{17}$ \\
$9_{29}$ \\
$9_{32}$ \\
$9_{33}$ \\
$9_{38}$ \\
$10_{79}$ \\
$10_{80}$ \\
$10_{81}$ \\
$10_{82}$ \\
$10_{83}$ \\
$10_{84}$ \\
$10_{85}$ \\
$10_{86}$ \\
$10_{87}$ \\
$10_{88}$ \\
$10_{89}$ \\
$10_{90}$ \\
$10_{91}$ \\
$10_{92}$ \\
$10_{93}$ \\
$10_{94}$ \\
$10_{95}$ \\
$10_{96}$ \\
$10_{97}$ \\
$10_{98}$ \\
$10_{99}$ \\
$10_{109}$ \\
$10_{111}$ \\
$10_{116}$ \\
$10_{117}$ \\
$10_{122}$ \\
$10_{123}$ \\
$10_{148}$ \\
$10_{149}$ \\
$10_{150}$ \\
$10_{151}$ \\
$10_{152}$ \\
$10_{153}$ \\
$10_{154}$ \\
$11a_{2}$ \\
$11a_{3}$ \\
$11a_{14}$ \\
$11a_{15}$ \\
$11a_{17}$ \\
$11a_{18}$ \\
$11a_{19}$ \\
$11a_{20}$ \\
$11a_{22}$ \\
$11a_{24}$ \\
$11a_{25}$ \\
$11a_{26}$ \\
$11a_{27}$ \\
$11a_{28}$ \\
$11a_{29}$ \\
$11a_{30}$ \\
$11a_{38}$ \\
$11a_{43}$ \\
$11a_{44}$ \\
$11a_{47}$ \\
$11a_{49}$ \\
$11a_{52}$ \\
$11a_{53}$ \\
$11a_{54}$ \\
$11a_{57}$ \\
$11a_{66}$ \\
$11a_{67}$ \\
$11a_{68}$ \\
$11a_{69}$ \\
$11a_{70}$ \\
$11a_{71}$ \\
$11a_{72}$ \\
$11a_{76}$ \\
$11a_{79}$ \\
$11a_{80}$ \\
$11a_{81}$ \\
$11a_{102}$ \\
$11a_{123}$ \\
$11a_{124}$ \\
$11a_{126}$ \\
$11a_{127}$ \\
$11a_{129}$ \\
$11a_{130}$ \\
$11a_{131}$ \\
$11a_{132}$ \\
$11a_{138}$ \\
$11a_{141}$ \\
$11a_{147}$ \\
$11a_{149}$ \\
$11a_{150}$ \\
$11a_{151}$ \\
$11a_{152}$ \\
$11a_{156}$ \\
$11a_{157}$ \\
$11a_{164}$ \\
$11a_{167}$ \\
$11a_{172}$ \\
$11a_{173}$ \\
$11a_{231}$ \\
$11a_{232}$ \\
$11a_{244}$ \\
$11a_{250}$ \\
$11a_{251}$ \\
$11a_{252}$ \\
$11a_{253}$ \\
$11a_{254}$ \\
$11a_{261}$ \\
$11a_{262}$ \\
$11a_{263}$ \\
$11a_{264}$ \\
$11a_{265}$ \\
$11a_{266}$ \\
$11a_{267}$ \\
$11a_{269}$ \\
$11a_{273}$ \\
$11a_{274}$ \\
$11a_{275}$ \\
$11a_{281}$ \\
$11a_{284}$ \\
$11a_{286}$ \\
$11a_{287}$ \\
$11a_{288}$ \\
$11a_{290}$ \\
$11a_{291}$ \\
$11a_{292}$ \\
$11a_{293}$ \\
$11a_{294}$ \\
$11a_{298}$ \\
$11a_{299}$ \\
$11a_{300}$ \\
$11a_{301}$ \\
$11a_{304}$ \\
$11a_{305}$ \\
$11a_{314}$ \\
$11a_{315}$ \\
$11a_{316}$ \\
$11a_{323}$ \\
$11a_{327}$ \\
$11a_{328}$ \\
$11a_{332}$ \\
$11a_{346}$ \\
$11a_{347}$ \\
$11a_{350}$ \\
$11a_{353}$ \\
$11a_{354}$ \\
$11n_{4}$ \\
$11n_{5}$ \\
$11n_{6}$ \\
$11n_{7}$ \\
$11n_{8}$ \\
$11n_{9}$ \\
$11n_{10}$ \\
$11n_{11}$ \\
$11n_{21}$ \\
$11n_{22}$ \\
$11n_{23}$ \\
$11n_{24}$ \\
$11n_{25}$ \\
$11n_{26}$ \\
$11n_{27}$ \\
$11n_{31}$ \\
$11n_{32}$ \\
$11n_{33}$ \\
$11n_{34}$ \\
$11n_{35}$ \\
$11n_{36}$ \\
$11n_{37}$ \\
$11n_{39}$ \\
$11n_{40}$ \\
$11n_{41}$ \\
$11n_{42}$ \\
$11n_{43}$ \\
$11n_{44}$ \\
$11n_{45}$ \\
$11n_{46}$ \\
$11n_{47}$ \\
$11n_{65}$ \\
$11n_{66}$ \\
$11n_{67}$ \\
$11n_{68}$ \\
$11n_{69}$ \\
$11n_{71}$ \\
$11n_{72}$ \\
$11n_{73}$ \\
$11n_{74}$ \\
$11n_{75}$ \\
$11n_{76}$ \\
$11n_{77}$ \\
$11n_{78}$ \\
$11n_{80}$ \\
$11n_{81}$ \\
$11n_{97}$ \\
$11n_{98}$ \\
$11n_{99}$ \\
$11n_{151}$ \\
$11n_{152}$ \\
$11n_{156}$ \\
$11n_{160}$ \\
$11n_{166}$ \\
$11n_{182}$ \\
$12a_{0001}$ \\
$12a_{0002}$ \\
$12a_{0004}$ \\
$12a_{0005}$ \\
$12a_{0006}$ \\
$12a_{0007}$ \\
$12a_{0008}$ \\
$12a_{0010}$ \\
$12a_{0011}$ \\
$12a_{0013}$ \\
$12a_{0014}$ \\
$12a_{0015}$ \\
$12a_{0023}$ \\
$12a_{0029}$ \\
$12a_{0030}$ \\
$12a_{0033}$ \\
$12a_{0036}$ \\
$12a_{0039}$ \\
$12a_{0040}$ \\
$12a_{0041}$ \\
$12a_{0043}$ \\
$12a_{0044}$ \\
$12a_{0045}$ \\
$12a_{0046}$ \\
$12a_{0047}$ \\
$12a_{0048}$ \\
$12a_{0049}$ \\
$12a_{0050}$ \\
$12a_{0051}$ \\
$12a_{0052}$ \\
$12a_{0053}$ \\
$12a_{0054}$ \\
$12a_{0055}$ \\
$12a_{0057}$ \\
$12a_{0058}$ \\
$12a_{0059}$ \\
$12a_{0060}$ \\
$12a_{0061}$ \\
$12a_{0063}$ \\
$12a_{0064}$ \\
$12a_{0065}$ \\
$12a_{0066}$ \\
$12a_{0067}$ \\
$12a_{0068}$ \\
$12a_{0069}$ \\
$12a_{0070}$ \\
$12a_{0071}$ \\
$12a_{0072}$ \\
$12a_{0073}$ \\
$12a_{0074}$ \\
$12a_{0075}$ \\
$12a_{0076}$ \\
$12a_{0079}$ \\
$12a_{0080}$ \\
$12a_{0088}$ \\
$12a_{0089}$ \\
$12a_{0090}$ \\
$12a_{0091}$ \\
$12a_{0092}$ \\
$12a_{0093}$ \\
$12a_{0094}$ \\
$12a_{0100}$ \\
$12a_{0101}$ \\
$12a_{0102}$ \\
$12a_{0103}$ \\
$12a_{0105}$ \\
$12a_{0107}$ \\
$12a_{0108}$ \\
$12a_{0109}$ \\
$12a_{0111}$ \\
$12a_{0113}$ \\
$12a_{0114}$ \\
$12a_{0115}$ \\
$12a_{0116}$ \\
$12a_{0117}$ \\
$12a_{0119}$ \\
$12a_{0120}$ \\
$12a_{0122}$ \\
$12a_{0123}$ \\
$12a_{0125}$ \\
$12a_{0126}$ \\
$12a_{0127}$ \\
$12a_{0129}$ \\
$12a_{0131}$ \\
$12a_{0132}$ \\
$12a_{0133}$ \\
$12a_{0134}$ \\
$12a_{0135}$ \\
$12a_{0136}$ \\
$12a_{0137}$ \\
$12a_{0138}$ \\
$12a_{0139}$ \\
$12a_{0140}$ \\
$12a_{0150}$ \\
$12a_{0154}$ \\
$12a_{0155}$ \\
$12a_{0156}$ \\
$12a_{0157}$ \\
$12a_{0162}$ \\
$12a_{0163}$ \\
$12a_{0164}$ \\
$12a_{0166}$ \\
$12a_{0167}$ \\
$12a_{0177}$ \\
$12a_{0182}$ \\
$12a_{0184}$ \\
$12a_{0185}$ \\
$12a_{0186}$ \\
$12a_{0187}$ \\
$12a_{0188}$ \\
$12a_{0189}$ \\
$12a_{0190}$ \\
$12a_{0191}$ \\
$12a_{0192}$ \\
$12a_{0195}$ \\
$12a_{0198}$ \\
$12a_{0199}$ \\
$12a_{0200}$ \\
$12a_{0201}$ \\
$12a_{0207}$ \\
$12a_{0211}$ \\
$12a_{0213}$ \\
$12a_{0214}$ \\
$12a_{0224}$ \\
$12a_{0227}$ \\
$12a_{0228}$ \\
$12a_{0231}$ \\
$12a_{0233}$ \\
$12a_{0244}$ \\
$12a_{0245}$ \\
$12a_{0260}$ \\
$12a_{0266}$ \\
$12a_{0268}$ \\
$12a_{0269}$ \\
$12a_{0271}$ \\
$12a_{0273}$ \\
$12a_{0276}$ \\
$12a_{0282}$ \\
$12a_{0283}$ \\
$12a_{0286}$ \\
$12a_{0288}$ \\
$12a_{0290}$ \\
$12a_{0291}$ \\
$12a_{0311}$ \\
$12a_{0312}$ \\
$12a_{0313}$ \\
$12a_{0315}$ \\
$12a_{0318}$ \\
$12a_{0328}$ \\
$12a_{0331}$ \\
$12a_{0333}$ \\
$12a_{0334}$ \\
$12a_{0336}$ \\
$12a_{0337}$ \\
$12a_{0341}$ \\
$12a_{0343}$ \\
$12a_{0344}$ \\
$12a_{0348}$ \\
$12a_{0350}$ \\
$12a_{0351}$ \\
$12a_{0353}$ \\
$12a_{0355}$ \\
$12a_{0361}$ \\
$12a_{0365}$ \\
$12a_{0366}$ \\
$12a_{0377}$ \\
$12a_{0389}$ \\
$12a_{0390}$ \\
$12a_{0396}$ \\
$12a_{0398}$ \\
$12a_{0399}$ \\
$12a_{0407}$ \\
$12a_{0409}$ \\
$12a_{0410}$ \\
$12a_{0411}$ \\
$12a_{0412}$ \\
$12a_{0414}$ \\
$12a_{0416}$ \\
$12a_{0417}$ \\
$12a_{0427}$ \\
$12a_{0428}$ \\
$12a_{0429}$ \\
$12a_{0430}$ \\
$12a_{0432}$ \\
$12a_{0434}$ \\
$12a_{0435}$ \\
$12a_{0439}$ \\
$12a_{0440}$ \\
$12a_{0441}$ \\
$12a_{0446}$ \\
$12a_{0452}$ \\
$12a_{0455}$ \\
$12a_{0456}$ \\
$12a_{0460}$ \\
$12a_{0462}$ \\
$12a_{0465}$ \\
$12a_{0466}$ \\
$12a_{0468}$ \\
$12a_{0475}$ \\
$12a_{0479}$ \\
$12a_{0483}$ \\
$12a_{0484}$ \\
$12a_{0489}$ \\
$12a_{0490}$ \\
$12a_{0491}$ \\
$12a_{0493}$ \\
$12a_{0494}$ \\
$12a_{0495}$ \\
$12a_{0509}$ \\
$12a_{0523}$ \\
$12a_{0526}$ \\
$12a_{0527}$ \\
$12a_{0554}$ \\
$12a_{0556}$ \\
$12a_{0567}$ \\
$12a_{0587}$ \\
$12a_{0588}$ \\
$12a_{0589}$ \\
$12a_{0590}$ \\
$12a_{0598}$ \\
$12a_{0599}$ \\
$12a_{0605}$ \\
$12a_{0606}$ \\
$12a_{0612}$ \\
$12a_{0613}$ \\
$12a_{0617}$ \\
$12a_{0620}$ \\
$12a_{0623}$ \\
$12a_{0624}$ \\
$12a_{0627}$ \\
$12a_{0629}$ \\
$12a_{0630}$ \\
$12a_{0633}$ \\
$12a_{0634}$ \\
$12a_{0637}$ \\
$12a_{0638}$ \\
$12a_{0639}$ \\
$12a_{0640}$ \\
$12a_{0645}$ \\
$12a_{0647}$ \\
$12a_{0654}$ \\
$12a_{0655}$ \\
$12a_{0657}$ \\
$12a_{0658}$ \\
$12a_{0659}$ \\
$12a_{0667}$ \\
$12a_{0668}$ \\
$12a_{0672}$ \\
$12a_{0675}$ \\
$12a_{0676}$ \\
$12a_{0680}$ \\
$12a_{0685}$ \\
$12a_{0688}$ \\
$12a_{0692}$ \\
$12a_{0693}$ \\
$12a_{0694}$ \\
$12a_{0697}$ \\
$12a_{0698}$ \\
$12a_{0699}$ \\
$12a_{0701}$ \\
$12a_{0702}$ \\
$12a_{0703}$ \\
$12a_{0705}$ \\
$12a_{0706}$ \\
$12a_{0707}$ \\
$12a_{0708}$ \\
$12a_{0709}$ \\
$12a_{0710}$ \\
$12a_{0741}$ \\
$12a_{0750}$ \\
$12a_{0755}$ \\
$12a_{0771}$ \\
$12a_{0798}$ \\
$12a_{0801}$ \\
$12a_{0804}$ \\
$12a_{0811}$ \\
$12a_{0812}$ \\
$12a_{0813}$ \\
$12a_{0814}$ \\
$12a_{0815}$ \\
$12a_{0816}$ \\
$12a_{0817}$ \\
$12a_{0818}$ \\
$12a_{0824}$ \\
$12a_{0825}$ \\
$12a_{0828}$ \\
$12a_{0829}$ \\
$12a_{0830}$ \\
$12a_{0831}$ \\
$12a_{0832}$ \\
$12a_{0833}$ \\
$12a_{0834}$ \\
$12a_{0841}$ \\
$12a_{0844}$ \\
$12a_{0845}$ \\
$12a_{0846}$ \\
$12a_{0847}$ \\
$12a_{0848}$ \\
$12a_{0849}$ \\
$12a_{0850}$ \\
$12a_{0851}$ \\
$12a_{0852}$ \\
$12a_{0853}$ \\
$12a_{0854}$ \\
$12a_{0859}$ \\
$12a_{0860}$ \\
$12a_{0861}$ \\
$12a_{0862}$ \\
$12a_{0863}$ \\
$12a_{0866}$ \\
$12a_{0867}$ \\
$12a_{0868}$ \\
$12a_{0871}$ \\
$12a_{0872}$ \\
$12a_{0873}$ \\
$12a_{0874}$ \\
$12a_{0875}$ \\
$12a_{0884}$ \\
$12a_{0885}$ \\
$12a_{0886}$ \\
$12a_{0888}$ \\
$12a_{0891}$ \\
$12a_{0893}$ \\
$12a_{0894}$ \\
$12a_{0895}$ \\
$12a_{0898}$ \\
$12a_{0899}$ \\
$12a_{0900}$ \\
$12a_{0903}$ \\
$12a_{0906}$ \\
$12a_{0909}$ \\
$12a_{0910}$ \\
$12a_{0911}$ \\
$12a_{0912}$ \\
$12a_{0913}$ \\
$12a_{0914}$ \\
$12a_{0916}$ \\
$12a_{0917}$ \\
$12a_{0923}$ \\
$12a_{0924}$ \\
$12a_{0930}$ \\
$12a_{0931}$ \\
$12a_{0933}$ \\
$12a_{0934}$ \\
$12a_{0936}$ \\
$12a_{0939}$ \\
$12a_{0940}$ \\
$12a_{0941}$ \\
$12a_{0942}$ \\
$12a_{0944}$ \\
$12a_{0945}$ \\
$12a_{0946}$ \\
$12a_{0947}$ \\
$12a_{0948}$ \\
$12a_{0949}$ \\
$12a_{0951}$ \\
$12a_{0952}$ \\
$12a_{0953}$ \\
$12a_{0956}$ \\
$12a_{0957}$ \\
$12a_{0958}$ \\
$12a_{0959}$ \\
$12a_{0960}$ \\
$12a_{0961}$ \\
$12a_{0964}$ \\
$12a_{0965}$ \\
$12a_{0966}$ \\
$12a_{0970}$ \\
$12a_{0971}$ \\
$12a_{0975}$ \\
$12a_{0976}$ \\
$12a_{0979}$ \\
$12a_{0981}$ \\
$12a_{0982}$ \\
$12a_{0986}$ \\
$12a_{0987}$ \\
$12a_{0990}$ \\
$12a_{0994}$ \\
$12a_{0999}$ \\
$12a_{1000}$ \\
$12a_{1002}$ \\
$12a_{1003}$ \\
$12a_{1004}$ \\
$12a_{1007}$ \\
$12a_{1008}$ \\
$12a_{1009}$ \\
$12a_{1010}$ \\
$12a_{1011}$ \\
$12a_{1012}$ \\
$12a_{1013}$ \\
$12a_{1014}$ \\
$12a_{1015}$ \\
$12a_{1016}$ \\
$12a_{1017}$ \\
$12a_{1018}$ \\
$12a_{1019}$ \\
$12a_{1020}$ \\
$12a_{1021}$ \\
$12a_{1036}$ \\
$12a_{1038}$ \\
$12a_{1043}$ \\
$12a_{1050}$ \\
$12a_{1053}$ \\
$12a_{1057}$ \\
$12a_{1059}$ \\
$12a_{1061}$ \\
$12a_{1062}$ \\
$12a_{1064}$ \\
$12a_{1065}$ \\
$12a_{1067}$ \\
$12a_{1070}$ \\
$12a_{1074}$ \\
$12a_{1076}$ \\
$12a_{1078}$ \\
$12a_{1079}$ \\
$12a_{1080}$ \\
$12a_{1081}$ \\
$12a_{1083}$ \\
$12a_{1087}$ \\
$12a_{1088}$ \\
$12a_{1091}$ \\
$12a_{1092}$ \\
$12a_{1093}$ \\
$12a_{1095}$ \\
$12a_{1096}$ \\
$12a_{1097}$ \\
$12a_{1098}$ \\
$12a_{1099}$ \\
$12a_{1100}$ \\
$12a_{1101}$ \\
$12a_{1102}$ \\
$12a_{1103}$ \\
$12a_{1105}$ \\
$12a_{1109}$ \\
$12a_{1110}$ \\
$12a_{1111}$ \\
$12a_{1117}$ \\
$12a_{1119}$ \\
$12a_{1120}$ \\
$12a_{1121}$ \\
$12a_{1122}$ \\
$12a_{1123}$ \\
$12a_{1124}$ \\
$12a_{1156}$ \\
$12a_{1167}$ \\
$12a_{1173}$ \\
$12a_{1175}$ \\
$12a_{1181}$ \\
$12a_{1184}$ \\
$12a_{1185}$ \\
$12a_{1186}$ \\
$12a_{1187}$ \\
$12a_{1190}$ \\
$12a_{1191}$ \\
$12a_{1192}$ \\
$12a_{1193}$ \\
$12a_{1194}$ \\
$12a_{1196}$ \\
$12a_{1202}$ \\
$12a_{1203}$ \\
$12a_{1204}$ \\
$12a_{1205}$ \\
$12a_{1209}$ \\
$12a_{1211}$ \\
$12a_{1212}$ \\
$12a_{1213}$ \\
$12a_{1215}$ \\
$12a_{1216}$ \\
$12a_{1217}$ \\
$12a_{1218}$ \\
$12a_{1219}$ \\
$12a_{1220}$ \\
$12a_{1221}$ \\
$12a_{1222}$ \\
$12a_{1225}$ \\
$12a_{1229}$ \\
$12a_{1230}$ \\
$12a_{1231}$ \\
$12a_{1232}$ \\
$12a_{1237}$ \\
$12a_{1246}$ \\
$12a_{1248}$ \\
$12a_{1249}$ \\
$12a_{1251}$ \\
$12a_{1252}$ \\
$12a_{1253}$ \\
$12a_{1254}$ \\
$12a_{1255}$ \\
$12a_{1256}$ \\
$12a_{1257}$ \\
$12a_{1260}$ \\
$12a_{1261}$ \\
$12a_{1263}$ \\
$12a_{1266}$ \\
$12a_{1267}$ \\
$12a_{1269}$ \\
$12a_{1270}$ \\
$12a_{1271}$ \\
$12a_{1272}$ \\
$12a_{1288}$ \\
$12n_{0001}$ \\
$12n_{0002}$ \\
$12n_{0003}$ \\
$12n_{0004}$ \\
$12n_{0005}$ \\
$12n_{0006}$ \\
$12n_{0007}$ \\
$12n_{0008}$ \\
$12n_{0009}$ \\
$12n_{0010}$ \\
$12n_{0014}$ \\
$12n_{0015}$ \\
$12n_{0016}$ \\
$12n_{0017}$ \\
$12n_{0018}$ \\
$12n_{0019}$ \\
$12n_{0020}$ \\
$12n_{0021}$ \\
$12n_{0022}$ \\
$12n_{0023}$ \\
$12n_{0024}$ \\
$12n_{0026}$ \\
$12n_{0027}$ \\
$12n_{0028}$ \\
$12n_{0029}$ \\
$12n_{0030}$ \\
$12n_{0031}$ \\
$12n_{0032}$ \\
$12n_{0033}$ \\
$12n_{0034}$ \\
$12n_{0049}$ \\
$12n_{0050}$ \\
$12n_{0051}$ \\
$12n_{0052}$ \\
$12n_{0053}$ \\
$12n_{0055}$ \\
$12n_{0056}$ \\
$12n_{0057}$ \\
$12n_{0058}$ \\
$12n_{0059}$ \\
$12n_{0060}$ \\
$12n_{0061}$ \\
$12n_{0062}$ \\
$12n_{0063}$ \\
$12n_{0064}$ \\
$12n_{0066}$ \\
$12n_{0067}$ \\
$12n_{0068}$ \\
$12n_{0069}$ \\
$12n_{0070}$ \\
$12n_{0071}$ \\
$12n_{0072}$ \\
$12n_{0073}$ \\
$12n_{0074}$ \\
$12n_{0075}$ \\
$12n_{0076}$ \\
$12n_{0080}$ \\
$12n_{0081}$ \\
$12n_{0082}$ \\
$12n_{0083}$ \\
$12n_{0084}$ \\
$12n_{0085}$ \\
$12n_{0086}$ \\
$12n_{0087}$ \\
$12n_{0088}$ \\
$12n_{0089}$ \\
$12n_{0090}$ \\
$12n_{0091}$ \\
$12n_{0092}$ \\
$12n_{0093}$ \\
$12n_{0094}$ \\
$12n_{0095}$ \\
$12n_{0096}$ \\
$12n_{0097}$ \\
$12n_{0098}$ \\
$12n_{0099}$ \\
$12n_{0100}$ \\
$12n_{0101}$ \\
$12n_{0102}$ \\
$12n_{0103}$ \\
$12n_{0104}$ \\
$12n_{0105}$ \\
$12n_{0106}$ \\
$12n_{0107}$ \\
$12n_{0108}$ \\
$12n_{0109}$ \\
$12n_{0110}$ \\
$12n_{0111}$ \\
$12n_{0112}$ \\
$12n_{0113}$ \\
$12n_{0114}$ \\
$12n_{0115}$ \\
$12n_{0116}$ \\
$12n_{0117}$ \\
$12n_{0118}$ \\
$12n_{0119}$ \\
$12n_{0120}$ \\
$12n_{0122}$ \\
$12n_{0123}$ \\
$12n_{0124}$ \\
$12n_{0125}$ \\
$12n_{0126}$ \\
$12n_{0127}$ \\
$12n_{0128}$ \\
$12n_{0129}$ \\
$12n_{0130}$ \\
$12n_{0131}$ \\
$12n_{0132}$ \\
$12n_{0133}$ \\
$12n_{0134}$ \\
$12n_{0135}$ \\
$12n_{0136}$ \\
$12n_{0137}$ \\
$12n_{0138}$ \\
$12n_{0139}$ \\
$12n_{0140}$ \\
$12n_{0141}$ \\
$12n_{0156}$ \\
$12n_{0157}$ \\
$12n_{0158}$ \\
$12n_{0173}$ \\
$12n_{0174}$ \\
$12n_{0175}$ \\
$12n_{0176}$ \\
$12n_{0177}$ \\
$12n_{0178}$ \\
$12n_{0179}$ \\
$12n_{0180}$ \\
$12n_{0181}$ \\
$12n_{0182}$ \\
$12n_{0183}$ \\
$12n_{0184}$ \\
$12n_{0185}$ \\
$12n_{0186}$ \\
$12n_{0187}$ \\
$12n_{0188}$ \\
$12n_{0189}$ \\
$12n_{0190}$ \\
$12n_{0191}$ \\
$12n_{0192}$ \\
$12n_{0193}$ \\
$12n_{0194}$ \\
$12n_{0195}$ \\
$12n_{0196}$ \\
$12n_{0197}$ \\
$12n_{0201}$ \\
$12n_{0202}$ \\
$12n_{0203}$ \\
$12n_{0204}$ \\
$12n_{0205}$ \\
$12n_{0206}$ \\
$12n_{0207}$ \\
$12n_{0208}$ \\
$12n_{0209}$ \\
$12n_{0210}$ \\
$12n_{0211}$ \\
$12n_{0212}$ \\
$12n_{0213}$ \\
$12n_{0214}$ \\
$12n_{0215}$ \\
$12n_{0216}$ \\
$12n_{0217}$ \\
$12n_{0219}$ \\
$12n_{0220}$ \\
$12n_{0221}$ \\
$12n_{0222}$ \\
$12n_{0223}$ \\
$12n_{0224}$ \\
$12n_{0225}$ \\
$12n_{0226}$ \\
$12n_{0227}$ \\
$12n_{0228}$ \\
$12n_{0229}$ \\
$12n_{0230}$ \\
$12n_{0231}$ \\
$12n_{0232}$ \\
$12n_{0245}$ \\
$12n_{0246}$ \\
$12n_{0247}$ \\
$12n_{0252}$ \\
$12n_{0253}$ \\
$12n_{0254}$ \\
$12n_{0255}$ \\
$12n_{0256}$ \\
$12n_{0257}$ \\
$12n_{0258}$ \\
$12n_{0259}$ \\
$12n_{0260}$ \\
$12n_{0261}$ \\
$12n_{0262}$ \\
$12n_{0263}$ \\
$12n_{0264}$ \\
$12n_{0265}$ \\
$12n_{0266}$ \\
$12n_{0267}$ \\
$12n_{0268}$ \\
$12n_{0269}$ \\
$12n_{0270}$ \\
$12n_{0290}$ \\
$12n_{0291}$ \\
$12n_{0292}$ \\
$12n_{0315}$ \\
$12n_{0317}$ \\
$12n_{0320}$ \\
$12n_{0326}$ \\
$12n_{0329}$ \\
$12n_{0330}$ \\
$12n_{0331}$ \\
$12n_{0335}$ \\
$12n_{0336}$ \\
$12n_{0344}$ \\
$12n_{0345}$ \\
$12n_{0346}$ \\
$12n_{0364}$ \\
$12n_{0365}$ \\
$12n_{0376}$ \\
$12n_{0421}$ \\
$12n_{0422}$ \\
$12n_{0423}$ \\
$12n_{0431}$ \\
$12n_{0437}$ \\
$12n_{0440}$ \\
$12n_{0448}$ \\
$12n_{0454}$ \\
$12n_{0455}$ \\
$12n_{0456}$ \\
$12n_{0484}$ \\
$12n_{0485}$ \\
$12n_{0494}$ \\
$12n_{0495}$ \\
$12n_{0496}$ \\
$12n_{0508}$ \\
$12n_{0518}$ \\
$12n_{0526}$ \\
$12n_{0533}$ \\
$12n_{0538}$ \\
$12n_{0539}$ \\
$12n_{0541}$ \\
$12n_{0548}$ \\
$12n_{0553}$ \\
$12n_{0554}$ \\
$12n_{0555}$ \\
$12n_{0556}$ \\
$12n_{0558}$ \\
$12n_{0567}$ \\
$12n_{0568}$ \\
$12n_{0600}$ \\
$12n_{0601}$ \\
$12n_{0602}$ \\
$12n_{0604}$ \\
$12n_{0605}$ \\
$12n_{0622}$ \\
$12n_{0633}$ \\
$12n_{0642}$ \\
$12n_{0658}$ \\
$12n_{0665}$ \\
$12n_{0670}$ \\
$12n_{0671}$ \\
$12n_{0672}$ \\
$12n_{0673}$ \\
$12n_{0674}$ \\
$12n_{0675}$ \\
$12n_{0676}$ \\
$12n_{0677}$ \\
$12n_{0678}$ \\
$12n_{0679}$ \\
$12n_{0680}$ \\
$12n_{0681}$ \\
$12n_{0682}$ \\
$12n_{0688}$ \\
$12n_{0689}$ \\
$12n_{0690}$ \\
$12n_{0691}$ \\
$12n_{0692}$ \\
$12n_{0693}$ \\
$12n_{0694}$ \\
$12n_{0695}$ \\
$12n_{0696}$ \\
$12n_{0697}$ \\
$12n_{0702}$ \\
$12n_{0705}$ \\
$12n_{0706}$ \\
$12n_{0712}$ \\
$12n_{0714}$ \\
$12n_{0716}$ \\
$12n_{0728}$ \\
$12n_{0747}$ \\
$12n_{0748}$ \\
$12n_{0751}$ \\
$12n_{0753}$ \\
$12n_{0766}$ \\
$12n_{0769}$ \\
$12n_{0775}$ \\
$12n_{0776}$ \\
$12n_{0779}$ \\
$12n_{0780}$ \\
$12n_{0781}$ \\
$12n_{0782}$ \\
$12n_{0790}$ \\
$12n_{0800}$ \\
$12n_{0802}$ \\
$12n_{0826}$ \\
$12n_{0831}$ \\
$12n_{0834}$ \\
$12n_{0837}$ \\
$12n_{0840}$ \\
$12n_{0842}$ \\
$12n_{0853}$ \\
$12n_{0854}$ \\
$12n_{0863}$ \\
$12n_{0864}$ \\
$12n_{0866}$ \\
$12n_{0868}$ \\
$12n_{0869}$ \\
$12n_{0874}$ \\
$12n_{0879}$ \\
$12n_{0886}$ \\
$12n_{0887}$ \\
$12n_{0888}$
\end{multicols}
}

\end{document}